\documentclass[final]{siamart220329}

\usepackage{bm}
\usepackage{bbm}
\usepackage{amsfonts,amsmath,amssymb}
\usepackage{graphicx}
\usepackage{epstopdf}
\usepackage{mathtools}
\usepackage{graphicx}
\usepackage[shortlabels]{enumitem}
\usepackage{algorithm}
\usepackage[end]{algpseudocode}

\usepackage{amsopn}
\usepackage{booktabs}

\newtheorem{remark}[theorem]{Remark}
\newcommand{\TheTitle}{%
Reduced Krylov Basis Methods for Parametric Partial Differential Equations
}

\newcommand{\TheShortTitle}{%
Reduced Krylov Basis Methods
}

\newcommand{\TheAuthor}{%
  Y. Li, L. Zikatanov, C. Zuo
}

\newcommand{\TheFunding}{The work of Li is partially supported by the Fundamental Research Funds for the Central Universities (no. 226-2023-00039). 
The work of Zikatanov and Zuo is supported in part by NSF DMS-2208249. Zikatanov also acknowledges the support of the U.~S.-Norway Fulbright Foundation.  
}

\headers{\TheShortTitle}{\TheAuthor}
\title{\TheTitle
\thanks{\TheFunding}}

\author{Yuwen Li\thanks{School of Mathematical Sciences, Zhejiang University, Hangzhou, Zhejiang 310058, China 
  (\email{liyuwen@zju.edu.cn}).}
\and Ludmil T. Zikatanov\thanks{Department of Mathematics, The Pennsylvania State University,
University Park, PA 16802, USA
  (\email{ltz1@psu.edu}, \email{cxz5180@psu.edu}).}
\and Cheng Zuo\footnotemark[3]}

\ifpdf%
\hypersetup{%
  pdftitle={\TheTitle},
  pdfauthor={\TheAuthor}
}
\fi

\begin{document}

\maketitle

%\begin{center}
%In collaboration with:
%  {\TheCollaborators}
%\end{center}
%\vspace{1cm}
% ---------------------------------------------
% ---------------------------------------------

\begin{abstract}
This work is on a user-friendly reduced basis method for solving a family of  parametric PDEs by preconditioned Krylov subspace methods including the conjugate gradient method, generalized minimum residual method, and bi-conjugate gradient method. 
The proposed methods use a preconditioned Krylov subspace method for a high-fidelity discretization of one parameter instance to generate orthogonal basis vectors of the reduced basis subspace. Then large-scale discrete parameter-dependent problems are approximately solved in the low-dimensional Krylov subspace.  As shown in the theory and experiments, only a small number of Krylov subspace iterations are needed to simultaneously generate  approximate solutions of a family of high-fidelity and large-scale systems in the reduced basis subspace. This reduces the computational cost dramatically because (1) to construct the reduced basis vectors, we only solve one large-scale problem in the high-fidelity level; and (2) the family of large-scale problems restricted to the reduced basis  subspace have much smaller sizes. 
\end{abstract}

\begin{keywords}
  % 7. Keywords that describe the paper
  reduced basis method, parametric PDEs,  conjugate gradient method, Krylov subspace method, generalized minimum residual method, preconditioning
\end{keywords}

\section{Introduction}\label{sec:intro} In real-world applications such as engineering optimal design, it is often necessary to solve Partial Differential Equations (PDEs)  dependent on varying parameter sequences, e.g., diffusion coefficients, elasticity parameters, viscosity constant, wave propagation speed etc. In  a Hilbert space $\mathcal{H}$,  linear parametric (discrete) PDEs are of the form $A(\mu)u_\mu=f_\mu$, where the solution $u_\mu$ continuously depends on $\mu$ and the parameter $\mu$ varies in a compact parameter set $\mathcal{P}$.
The \emph{Reduced Basis Method} (RBM) is the mainstream numerical method  for efficiently solving parametrized PDEs (cf.~\cite{Maday2006,DeVore2014,Quarteroni2016,HesthavenRozzaStamm2016}). In general, a standard RBM is divided into two stages. During the offline stage, a RBM constructs a small-scale offline reduced basis subspace $\mathcal{H}_m\subset\mathcal{H}$ with potentially expensive computational cost. Then the online stage of the RBM is able to efficiently compute reduced basis solutions $\hat{u}_\mu\approx u_\mu$ in the subspace $\mathcal{H}_m$ for a wide range of 
parameters $\mu$. The RBM could gain significant computational efficiency for many-query parameter-dependent problems provided the number of input parameters in the online stage is large. 

There are two families of classical algorithms for constructing the reduced basis subspace $\mathcal{H}_m$. The first one is the reduced basis greedy algorithm, which utilizes a posteriori error estimation and adaptive feedback loop to select snapshots at parameter instances $\mu_1, \mu_2,\ldots, \mu_m$. The resulting reduced basis subspace is of the form 
\[\mathcal{H}_m={\rm Span}\big\{u_{\mu_1}, u_{\mu_2},\ldots,u_{\mu_m}\big\}.
\]
Another popular offline algorithm is the data-dependent Proper Orthogonal Decomposition (POD). By singular value decomposition, the POD extracts the first $m$ principal components of available training data $u_{\tilde{\mu}_1}, u_{\tilde{\mu}_2},\ldots,u_{\tilde{\mu}_M}$ to span the reduced basis subspace. When $m\ll \dim(\mathcal{H})$, the RBM would be more efficient than directly solving the large-scale problem $A(\mu)u_\mu=f_\mu$ in $\mathcal{H}$ for a large number of $\mu$.

In this work, we propose a novel family of RBMs termed ``\emph{Reduced Kylov Basis Methods}" (RKBMs) based on \emph{Preconditioned Krylov subspace methods}. Fixing parameters $\mu_0, \mu_1$ and a preconditioner $B\approx A(\mu_0)^{-1}$, the offline module of a RKBM  runs $m$-step Krylov subspace method, e.g., the conjugate gradient, generalized minimum residual, or bi-conjugate gradient method for $BA(\mu_1)u_{\mu_1}=Bf$. Let $u_{\mu_1,0}, u_{\mu_1,1}, \ldots, u_{\mu_1,m-1}\approx u_{\mu_1}$ be the corresponding iterative solutions and $r_{j}=f-A(\mu_1)u_{\mu_1,j}$. The reduced basis subspace in the proposed RKBMs is set to be 
\[
\mathcal{H}_m={\rm Span}\big\{Br_0, Br_1, \ldots, Br_{m-1}\big\},
\]
which is indeed the same as the preconditioned Krylov subspace 
\[
\mathcal{K}_m^{\mu_1}:={\rm Span}\big\{Br_0, BA(\mu_1)Br_0, \ldots, (BA(\mu_1))^{m-1}Br_{m-1}\big\}.\] 

The distinct feature of RKBMs is the construction of the reduced basis subspace during the offline stage, which is different from the classical snapshot-based reduced basis algorithms. In particular, the proposed RKBMs are motivated by 
a key observation about the invariance of the Krylov subspace $\mathcal{K}_m^\mu$ with respect to $\mu$ under special parametric structures, i.e., $\mathcal{K}_m^\mu=\mathcal{K}^{\mu_1}_m$ for all parameters $\mu$. In these cases, the online reduced basis solution $\hat{u}_\mu$ turns out to be the iterative solution of a Krylov subspace method at the $m$-th step. As a consequence, we are able to theoretically justify our RKBMs using classical convergence estimates of preconditioned conjugate gradient, generalized minimum residual, and bi-conjugate gradient methods etc. In \cite{LiZikatanovZuo2024}, we used a special preconditioned conjugate gradient method to derive a highly efficient reduced-basis-type solver for fractional Laplacian.

As mentioned earlier, traditional RBMs  make use of solution snapshots at a set of parameter instances to construct $\mathcal{H}_m$. Among classical reduced basis algorithms, greedy-type RBMs are often validated by certified error bounds and solid convergence analysis (cf.~ \cite{RozzaHuynhPatera2008,BinevCohenDahmenDeVore2011,DeVorePetrova2013,CohenDeVore2015,LiSiegel2024}). However, greedy algorithms driven by a posteriori error estimation requires sophisticated implementation to achieve online efficiency. On the other hand, the performance of POD-based RBMs highly depend on the quality of the dataset, which is not available in many applications. In addition, the POD would be computationally expensive for large-scale datasets. In comparison to classical reduced basis approaches,  RKBMs simply set $\mathcal{H}_m$ to be the Krylov subspace, and thus its implementation is simple and similar to those well-known Krylov subpace iterative solvers. Therefore, our RKBMs are user-friendly and snapshot-free algorithms without data requirement. 

The rest of the paper is organized as follows. In Section \ref{secSPD}, we present the RCGBM and its convergence analysis for positive-definite problems with two parameters. In Section \ref{secNSPD}, we propose RKBMs and its convergence analysis for nonsymmetric and indefinite  problems with two parameters. We then discuss efficient variants  of {\rm RKBM}s for more general parametric structures in Section \ref{secmultipara}. To illustrate the effectiveness of the proposed RKBMs, we perform a  series of numerical experiments in Section \ref{secNE}.

\section{Parametric Elliptic Problems}\label{secSPD}
Let $\mathcal{H}$ be a Hilbert space and $\mathcal{H}^*$ be its dual space. Given a compact parameter set $\mathcal{P}\subset\mathbb{R}^d$ and a collection of functions $\{\theta_j: \mathcal{P}\rightarrow\mathbb{R}\}_{1\leq j\leq J}$ , we consider the parametrized bilinear form 
\begin{equation}\label{affineamu}
    a_{\mu}=\theta_1(\mu)a_1+\cdots+\theta_J(\mu)a_J,
\end{equation} 
where each $a_j: \mathcal{H}\times\mathcal{H}\rightarrow\mathbb{R}$ is symmetric and bilinear. In addition, we assume 
\begin{equation}\label{boundcoercive}
\begin{aligned}
    \sup_{0\neq v\in\mathcal{H}}\sup_{0\neq w\in\mathcal{H}}\frac{a_\mu(v,w)}{\|v\|_{\mathcal{H}}\|w\|_{\mathcal{H}}}\leq c_1,
        \quad\inf_{0\neq v\in\mathcal{H}}\frac{a_\mu(v,v)}{\|v\|_\mathcal{H}^2}\geq c_2
\end{aligned}
\end{equation}
for some constants $c_1>0$, $c_2>0$. Let $f\in\mathcal{H}^*$ be a bounded linear functional. 
By the Lax--Milgram theorem, there exists a unique $u_\mu\in\mathcal{H}$ such that
\begin{equation}\label{ellipticmodel}
    a_\mu(u_\mu,v)=f(v),\quad v\in\mathcal{H}.
\end{equation}
The parametric variational problem \eqref{ellipticmodel} is a model problem of RBMs, see, e.g., \cite{BinevCohenDahmenDeVore2011,Quarteroni2016}. In practice, $\mathcal{H}$ is a high-dimensional space, e.g., the \emph{finite element space} on a high-resolution mesh. The goal of RBMs is to efficiently find an approximate solution $\hat{u}_\mu$ for each parameter of interest $\mu$ without computing the high-fidelity solution $u_\mu$. 

When $\dim(\mathcal{H})=n<\infty$ is finite, let $\{\phi_i\}_{1\leq i\leq n}$ be a basis of $\mathcal{H}$. For simplicity of presentation, we rewrite \eqref{ellipticmodel} as the matrix-vector equation
\begin{equation}\label{Amuf}
    \mathbb{A}(\mu)\mathbf{u}_\mu:=\big(\theta_1(\mu)\mathbb{A}_1+\cdots+\theta_J(\mu)\mathbb{A}_J\big)\mathbf{u}_\mu=\mathbf{f},
\end{equation}
where 
$\mathbb{A}_j=(a_j(\phi_k,\phi_i))_{1\leq i,k\leq n}$, $\mathbf{f}=(f(\phi_1),\ldots,f(\phi_n))^t$. Here $\bullet^t$ denotes the transpose of a matrix. Then $(\phi_1,\ldots,\phi_n)\mathbf{u}_\mu$ gives rise to the solution $u_\mu$ of \eqref{ellipticmodel}.

Since $a_\mu$ and $\mathbb{A}(\mu)$ are symmetric and positive-definite (SPD), we shall use the Preconditioned Conjugate Gradient (PCG) method $\mathbb{B}\mathbb{A}(\mu_1)\mathbf{u}_{\mu_1}=\mathbb{B}\mathbf{f}$ at some parameter instance $\mu_1$ to generate a Krylov-type reduced basis subspace. The corresponding RBM is named as the Reduced Conjugate Gradient Basis Method (RCGBM). Let $(\bullet,\bullet)$ and $\|\bullet\|$ denote the Euclidean inner product and  norm, respectively. 
Our RCGBM for approximating the high-fidelity solution $u_\mu$ or $\mathbf{u}_\mu$ is described in Algorithm \ref{a:RCGBM}.
\begin{algorithm}[H]
  \caption{${\rm RCGBM}(\mathbb{B}, \mathcal{P}, \mu_0, \mu_1,  m)$ for solving $\mathbb{A}(\mu)\mathbf{u}_\mu=\mathbf{f}$}\label{a:RCGBM}
  \begin{algorithmic}
    \State \textbf{Input:}   a parameter set $\mathcal{P}$ of practical interest, $\mu_0\neq\mu_1\in\mathcal{P}$, an integer $m>0$, a preconditioner $\mathbb{B}: \mathbb{R}^n\rightarrow\mathbb{R}^n$ for $\mathbb{A}(\mu_0)$.
    \State\textbf{Offline Stage:} set 
    $\mathbf{u}_0=\mathbf{0}$,
    $\mathbf{r}_0=\mathbf{f}$,  $\mathbf{p}_0=\mathbf{z}_0=\mathbb{B}\mathbf{r}_0;$ 
   \For{$k=1:m-1$}
   \State $\alpha_k=(\mathbf{r}_{k-1},\mathbf{z}_{k-1})/(\mathbb{A}(\mu_1)\mathbf{p}_{k-1},\mathbf{p}_{k-1})$;
   \State $\mathbf{u}_k =  \mathbf{u}_{k-1} + \alpha_k\mathbf{p}_{k-1}$; 
   \State $\mathbf{r}_k =  \mathbf{r}_{k-1} - \alpha_k\mathbb{A}(\mu_1)\mathbf{p}_{k-1}$; 
   \State $\mathbf{z}_{k} = \mathbb{B} \mathbf{r}_{k}$; 
   \State $\beta_k = (\mathbf{r}_{k},z_{k})/( \mathbf{r}_{k-1},\mathbf{z}_{k-1})$;
   \State $\mathbf{p}_k = \mathbf{z}_{k} + \beta_k\mathbf{p}_{k-1}$; 
   \EndFor
   \State construct the reduced basis subspace 
   \[\mathcal{K}_m:= {\rm Span}\big\{\mathbf{p}_0,\mathbf{p}_1,\dots,\mathbf{p}_{m-1}\big\}\]
   and the projection matrix 
   \begin{equation*}\mathbb{P}:= \left[\frac{\mathbf{p}_0}{\|\mathbf{p}_0\|},\frac{\mathbf{p}_1}{\|\mathbf{p}_1\|},\dots,\frac{\mathbf{p}_{m-1}}{\|\mathbf{p}_{m-1}\|}\right].\end{equation*}
   \State \textbf{Online Stage:} for each $\mu\in\mathcal{P}$, compute the reduced basis solution $\hat{\mathbf{u}}_\mu\in\mathcal{K}_m$ such that 
   \begin{equation*}
       (\mathbb{A}(\mu)\hat{\mathbf{u}}_\mu,\mathbf{v})=(\mathbf{f},\mathbf{v}),\quad \mathbf{v}\in\mathcal{K}_m.
   \end{equation*} 
Equivalently, $\hat{\mathbf{u}}_\mu=\mathbb{P}(\mathbb{P}^t\mathbb{A}(\mu)\mathbb{P})^{-1}\mathbb{P}^t\mathbf{f}$.   
  \end{algorithmic}  
\end{algorithm} 

Let $p_i=(\phi_1,\ldots,\phi_n)\mathbf{p}_i$ for $i=0, 1, \ldots, m-1$ and 
\[\mathcal{H}_m={\rm Span}\big\{p_0, p_1, \ldots, p_{m-1}\big\}.
\]
In Algorithm \ref{a:RCGBM}, $\hat{\mathbf{u}}_\mu$ is the coordinate vector of $\hat{u}_\mu\in\mathcal{H}_m$ determined by
\[
a_\mu(\hat{u}_\mu,v)=f(v),\quad v\in\mathcal{H}_m.
\]
The vector $\hat{\mathbf{u}}_\mu$ is computed by the formula
\begin{equation}\label{uhatmu}
\hat{\mathbf{u}}_\mu=\mathbb{P}(\mathbb{P}^t\mathbb{A}(\mu)\mathbb{P})^{-1}\mathbb{P}^t\mathbf{f}=\mathbb{P}\left(\sum_{j=1}^J\theta_j(\mu)\mathbb{P}^t\mathbb{A}_j\mathbb{P}\right)^{-1}\mathbb{P}^t\mathbf{f}.
\end{equation}
To ensure the efficiency of the RCGBM, the quantities  $\mathbb{P}^t\mathbb{A}_j\mathbb{P}$, $\mathbb{P}^t\mathbf{f}$ should be computed and stored prior to any online evaluation of $\hat{u}_\mu$ or $\hat{\mathbf{u}}_\mu$. In this way, the inversion process in \eqref{uhatmu} is equivalently to solving a dense but small-scale $m\times m$ 
 linear system. Therefore, for each parameter $\mu$, the cost of online stage in the RCGBM is expected to be smaller than directly assembling $\mathbb{A}(\mu)$, $\mathbf{f}$ and then solving $\mathbb{A}(\mu)\mathbf{u}_\mu=\mathbf{f}$ by classical nearly optimal solvers such as multigrid methods and fast Fourier transforms.

\begin{remark}
The RCGBM could be applied to 
the parametric problem
\begin{equation}\label{Amufmu}
    a_\mu(u_\mu,v)=f_\mu(v),\quad v\in\mathcal{H},
\end{equation}
with a parameter-dependent source $f_\mu\in\mathcal{H}^*$. For example, assume that $f_\mu$ admits the affine decomposition 
\begin{equation}\label{affinefmu}
f_\mu=\sum_{k=1}^K\omega_k(\mu)f_k.
\end{equation}
For each input parameter $\mu\in\mathcal{P}$, we use the RCGBM algorithm $K$ times to find $\hat{u}^k_\mu\in\mathcal{K}_m$ with $k=1, \ldots, K$ satisfying
\begin{equation*}
    a_\mu(\hat{u}^k_\mu,v)=f_k(v),\quad v\in\mathcal{K}_m.
\end{equation*}
Then the linear combination $$\hat{u}_\mu:=\sum_{k=1}^K\omega_k(\mu)\hat{u}_\mu^k$$ is an efficient and accurate reduced basis solution in \eqref{Amufmu}.
\end{remark}

When $a_\mu$ and $f_\mu$ are not of the forms \eqref{affineamu} and \eqref{affinefmu}, one could first use the Empirical Interpolation Method (cf.~\cite{BarraultMadayNguyen2004,MadayMulaTurinici2016,Li2024CGA}) to construct separable approximations $\bar{a}_\mu=\sum_{j=1}^J\bar{\theta}_j(\mu)\bar{a}_j\approx a_\mu$ and $\bar{f}_\mu=\sum_{k=1}^K\bar{\omega}_k(\mu)\bar{f}_k\approx f_\mu$. It is then natural to use the RCGBM to approximate the  perturbed  solution $\bar{u}_\mu\in\mathcal{H}$ given by 
\begin{equation*}
\bar{a}_\mu(\bar{u}_\mu,v)=\bar{f}_\mu(v),\quad v\in\mathcal{H}.    
\end{equation*}
In this way, one would achieve the same online efficiency  as the separable model \eqref{Amufmu}.

\subsection{Positive-Definite PDEs with Two Parameters}\label{subsectwopara}
First we consider the elliptic problem \eqref{ellipticmodel} with two parameters: find $u_{\mu}\in\mathcal{H}$ such that
\begin{equation}\label{twoparamodel}
a_\mu(u_\mu,v)=\theta_1(\mu)a_1(u_\mu,v)+\theta_2(\mu)a_2(u_\mu,v)=f(v),\quad v\in\mathcal{H}.
\end{equation}
The equivalent matrix-vector problem reads
\begin{equation}\label{twoparamodelmatrix}
\mathbb{A}(\mu)\mathbf{u}_\mu=\big(\theta_1(\mu)\mathbb{A}_1+\theta_2(\mu)\mathbb{A}_2\big)\mathbf{u}_\mu=\mathbf{f},
\end{equation}
where $\mathbb{A}_1, \mathbb{A}_2$ and $\mathbf{u}_\mu$  have the same meanings as in \eqref{Amuf}.
In the rest of this subsection, we present several important examples of parametric PDEs 
 that could be cast into \eqref{twoparamodel} and \eqref{twoparamodelmatrix}.

\subsubsection{Second Order Elliptic Equation}\label{subsecelliptic}
Let the domain $\Omega\subset\mathbb{R}^d$ be the union of non-overlapping open subdomains $\Omega_1$ and $\Omega_2$, i.e., $\overline{\Omega}=\overline{\Omega}_1\cup\overline{\Omega}_2$, $\Omega_1\cap\Omega_2=\emptyset$. Let $\mathcal{P}$ be a compact set in $\mathbb{R}_+\times\mathbb{R}_+$ and $\alpha_\mu=\nu_1\mathbbm{1}_{\Omega_1}+\nu_2\mathbbm{1}_{\Omega_2}$ for $\mu=(\nu_1,\nu_2)\in\mathcal{P}$. 
Let $\mathcal{T}_h$ be a conforming triangulation of $\Omega$ aligned with the interface $\overline{\Omega}_1\cap\overline{\Omega}_2$, and  $\mathcal{H}\subset H_0^1(\Omega)$ be a finite element space built upon $\mathcal{T}_h$. For $f\in L^2(\Omega)$, we use a Finite Element Method (FEM) to compute $u_\mu\in \mathcal{H}$ satisfying
\begin{equation}\label{twoparaelliptic}
     a_\mu(u_\mu,v)=\int_\Omega \alpha_\mu\nabla u_\mu\cdot\nabla vdx=\int_\Omega fvdx,\quad v\in\mathcal{H}.
\end{equation}
Clearly, \eqref{twoparaelliptic} is a discretization of the parametric diffusion equation $-\nabla\cdot(\alpha_\mu\nabla \tilde{u}_\mu)=f$
under the homogeneous Dirichlet boundary condition. In this case,
\begin{align*}
    a_j(v,w)&=\int_{\Omega_j} \nabla v\cdot\nabla wdx,\quad j=1, 2,\\
a_\mu&=\nu_1a_1+\nu_2a_2.
\end{align*}

\subsubsection{Linear Elasticity}\label{subseclinearelasticity}
Let the domain $\Omega$ be occupied by an elastic body. For $\bm{f}\in [L^2(\Omega)]^3$ we consider the the Navier--Lam\'e equation governing linear elasticity: 
\begin{align*}
    \frac{2\nu_1}{1+\nu_2}\nabla\cdot\varepsilon(\tilde{\bm{u}}_\mu)+\frac{\nu_1\nu_2}{(1+\nu_2)(1-2\nu_2)}\nabla(\nabla\cdot\tilde{\bm{u}}_\mu)=-\bm{f},
\end{align*}
where $\tilde{\bm{u}}_\mu\in[H_0^1(\Omega)]^3$ is the  displacement field, $\varepsilon(\bm{v})=(\nabla\bm{v}+\nabla\bm{v}^t)/2$, and $\mu=(\nu_1,\nu_2)$ with $\nu_1>0$, $\nu_2\in(0,\frac{1}{2})$ being Lam\'e constants.  
We define the parameter coefficients
\begin{align*}
    \theta_1(\mu):=\frac{\nu_1}{1+\nu_2},\quad\theta_2(\mu):=\frac{\nu_1\nu_2}{(1+\nu_2)(1-2\nu_2)}
\end{align*}
and $\mathcal{H}\subset[H_0^1(\Omega)]^3$ to be a conforming finite element space. The FEM for linear elasticity seeks $\bm{u}_\mu\in\mathcal{H}$ satisfying
\begin{equation}\label{linearElasticity}
a_\mu(\bm{u}_\mu,\bm{v})=\theta_1(\mu)a_1(\bm{u}_\mu,\bm{v})+\theta_2(\mu)a_2(\bm{u}_\mu,\bm{v})=\int_\Omega\bm{f}\cdot\bm{v}dx,\quad\bm{v}\in\mathcal{H},
\end{equation}  
where the bilinear forms $a_1, a_2$ are given by 
\begin{align*}
    a_1(\bm{v},\bm{w})&=\int_\Omega\varepsilon(\bm{v}):\varepsilon(\bm{w})dx,\\
    a_2(\bm{v},\bm{w})&=\int_\Omega(\nabla\cdot\bm{v})(\nabla\cdot\bm{w})dx.
\end{align*}

\subsubsection{Eddy Current Model}\label{subsecMaxwell} 
In the third example, we consider the eddy current model of electromagnetism 
\begin{align*}
    \nabla\times(\nu_1^{-1}\nabla\times\tilde{\bm{u}}_\mu)+\nu_2\tilde{\bm{u}}_\mu&=\bm{f}\quad\text{ in }\Omega,\\
\bm{n}\times(\tilde{\bm{u}}_\mu\times\bm{n})&=\bm{0}\quad\text{ on }\partial\Omega,
\end{align*}
where $\tilde{\bm{u}}_\mu$ is the electric field, $\bm{f}\in [L^2(\Omega)]^3$, $\bm{n}$ is the outward unit normal to $\partial\Omega$, and $\nu_1>0, \nu_2>0$ are material constants. We then define the Sobolev space
\begin{align*}
   H({\rm curl},\Omega)=\big\{\bm{v}\in L^2(\Omega): \nabla\times\bm{v}\in[L^2(\Omega)]^3\big\}  
\end{align*} 
and set $\mathcal{H}\subset H({\rm curl},\Omega)$ to be a N\'ed\'elec edge element space. Given $\mu=(\nu_1,\nu_2)$ with $\nu_1, \nu_2>0$, the FEM for the eddy current problem seeks $\bm{u}_\mu\in\mathcal{H}$ such that
\begin{equation*}
a_\mu(\bm{u}_\mu,\bm{v})=\theta_1(\mu)a_1(\bm{u}_\mu,\bm{v})+\theta_2(\mu)a_2(\bm{u}_\mu,\bm{v})=\int_\Omega\bm{f}\cdot\bm{v}dx,\quad\bm{v}\in\mathcal{H},
\end{equation*}
which is of the form \eqref{twoparamodel}, where $\theta_1(\mu)=\nu_1^{-1}$, $\theta_2(\mu)=\nu_2$, and
\begin{align*}
a_1(\bm{v},\bm{w})&=\int_\Omega(\nabla\times\bm{v})\cdot(\nabla\times\bm{w})dx,\\
    a_2(\bm{v},\bm{w})&=\int_\Omega\bm{v}\cdot\bm{w}dx.
\end{align*}

\subsection{Convergence of the RCGBM} In this subsection, we present a convergence analysis of the RCGBM for \eqref{twoparamodel} and \eqref{twoparamodelmatrix}. The next lemma is the key observation that is useful throughout the rest of this paper.
\begin{lemma}\label{invariancelemma}
Let $\mathbb{A}(\mu)=\theta_1(\mu)\mathbb{A}_1+\theta_2(\mu)\mathbb{A}_2$,  $\mathbb{B}=\mathbb{A}(\mu_0)^{-1}$ and
\begin{align*}
    &\mathcal{K}_m(\mathbb{B}\mathbb{A}(\mu),\mathbf{v}):={\rm Span}\big\{\mathbf{v}, \mathbb{B}\mathbb{A}(\mu)\mathbf{v}, \ldots,(\mathbb{B}\mathbb{A}(\mu))^{m-1}\mathbf{v}\big\}.
\end{align*}
Assume that $(\theta(\mu_i),\theta(\mu_i))$ and $(\theta(\mu_0),\theta(\mu_0))$ are linearly independent with $i=1, 2$. Then for any $m>0$ and  $\mathbf{v}\in\mathbb{R}^n$, we have
\begin{align*}
\mathcal{K}_m(\mathbb{B}\mathbb{A}(\mu_1),\mathbf{v})&=\mathcal{K}_m(\mathbb{B}\mathbb{A}(\mu_2),\mathbf{v}),\\
\mathcal{K}_m(\mathbb{B}^t\mathbb{A}(\mu_1)^t,\mathbf{v})&=\mathcal{K}_m(\mathbb{B}^t\mathbb{A}(\mu_2)^t,\mathbf{v}).
  \end{align*}
\end{lemma}
\begin{proof}
Let $\mathbb{I}$ denote the identity matrix and $c_1=\theta_1(\mu_0)$, $c_2=\theta_2(\mu_0)$. Without loss of generality, we assume that $c_1\neq0.$ It is observed that $c_1\mathbb{B}\mathbb{A}_1+c_2\mathbb{B}\mathbb{A}_2=\mathbb{I}$ and 
\begin{equation}\label{BAmu}
\begin{aligned}
 \mathbb{B}\mathbb{A}(\mu_i)&=\theta_1(\mu_i)\mathbb{B}\mathbb{A}_1+\theta_2(\mu_i)\mathbb{B}\mathbb{A}_2\\
    &=\frac{\theta_1(\mu_i)}{c_1}\mathbb{I}+\Big(\theta_2(\mu_i)-\theta_1(\mu_i)\frac{c_2}{c_1}\Big)\mathbb{B}\mathbb{A}_2,
    \end{aligned}
\end{equation}
where $i=1, 2.$
Note that $\theta_2(\mu_i)-\theta_1(\mu_i)c_2/c_1\neq0$ because $(\theta_1(\mu_i),\theta_2(\mu_i))$ and $(c_1,c_2)$ are not collinear. As a consequence of \eqref{BAmu}, it holds that
\begin{align*}
    \mathcal{K}_m(\mathbb{B}\mathbb{A}(\mu_1),\mathbf{v})&={\rm Span}\big\{\mathbf{v},\mathbb{B}\mathbb{A}(\mu_1)\mathbf{v}, \ldots,(\mathbb{B}\mathbb{A}(\mu_1))^{m-1}\mathbf{v}\big\}\\
    &={\rm Span}\big\{\mathbf{v}, \mathbb{B}\mathbb{A}_2\mathbf{v}, \ldots,(\mathbb{B}\mathbb{A}_2)^{m-1}\mathbf{v}\big\}\\
    &={\rm Span}\big\{\mathbf{v}, \mathbb{B}\mathbb{A}(\mu_2)\mathbf{v}, \ldots,(\mathbb{B}\mathbb{A}(\mu_2))^{m-1}\mathbf{v}\big\}=\mathcal{K}_m(\mathbb{B}\mathbb{A}(\mu_2),\mathbf{v}).
\end{align*}
The identity $\mathcal{K}_m(\mathbb{B}^t\mathbb{A}(\mu_1)^t,\mathbf{v})=\mathcal{K}_m(\mathbb{B}^t\mathbb{A}(\mu_2)^t,\mathbf{v})$ can be proved in a similar way.
\end{proof}

It turns out that the accuracy of $\hat{u}_\mu$ or $\hat{\mathbf{u}}_\mu$ can be readily derived via the classical convergence estimate of the conjugate gradient method. 
\begin{theorem}\label{errorRCGBMtwo}
Let $\mathbb{A}(\mu)=\theta_1(\mu)\mathbb{A}_1+\theta_2(\mu)\mathbb{A}_2$, $\mathbb{B}=\mathbb{A}(\mu_0)^{-1}$, $\kappa(\mu)=\kappa(\mathbb{B}\mathbb{A}(\mu))$. Assume $\theta_1(\mu_0)\theta_2(\mu_1)-\theta_2(\mu_0)\theta_1(\mu_1)\neq0$. Then for any $\mu\in\mathcal{P}$ with $\theta_1(\mu_0)\theta_2(\mu)-\theta_2(\mu_0)\theta_1(\mu)\neq0$, the reduced basis solution $\hat{\mathbf{u}}_\mu$ of ${\rm RCGBM}(\mathbb{B}, \mathcal{P}, \mu_0, \mu_1, m)$ satisfies
\begin{equation*}
      \|\mathbf{u}_\mu - \hat{\mathbf{u}}_\mu\|_{\mathbb{A}(\mu)}\leq 2\left(\frac{\sqrt{\kappa(\mu)} - 1}{\sqrt{\kappa(\mu)} + 1}\right)^{m}\|\mathbf{u}_\mu\|_{\mathbb{A}(\mu)}.
  \end{equation*}
\end{theorem}
\begin{proof}
It follows from Lemma \ref{invariancelemma} that
\begin{equation*}
    \mathcal{K}_m(\mathbb{B}\mathbb{A}(\mu),\mathbb{B}\mathbf{f})=\mathcal{K}_m,
\end{equation*}
where $\mathcal{K}_m=\mathcal{K}_m(\mathbb{B}\mathbb{A}(\mu_1),\mathbb{B}\mathbf{f})$ is the Krylov subspace in ${\rm RCGBM}(\mathbb{B}, \mathcal{P}, \mu_0, \mu_1, m)$. 
As a result,
the definition of $\hat{\mathbf{u}}_\mu$ implies that the energy-minimization property 
\begin{align*}
    \|\mathbf{u}_\mu-\hat{\mathbf{u}}_\mu\|_{\mathbb{A}(\mu)}&=\min_{\mathbf{v}\in\mathcal{K}_m}\|\mathbf{u}_\mu-\mathbf{v}\|_{\mathbb{A}(\mu)}\\
    &=\min_{\mathbf{v}\in\mathcal{K}_m(\mathbb{B}\mathbb{A}(\mu),\mathbb{B}\mathbf{f})}\|\mathbf{u}_\mu-\mathbf{v}\|_{\mathbb{A}(\mu)}.
\end{align*}
By the well-known property of the PCG, we observe that $\hat{\mathbf{u}}_\mu$ is also the iterative solution of the PCG for $\mathbb{B}\mathbb{A}(\mu) \mathbf{u}_\mu=\mathbb{B}\mathbf{f}$ at the $m$-th step. The classical convergence estimate of the PCG then implies 
\begin{equation*}
    \|\mathbf{u}_\mu-\hat{\mathbf{u}}_\mu\|_{\mathbb{A}(\mu)}\leq2\left(\frac{\sqrt{\kappa(\mu)} - 1}{\sqrt{\kappa(\mu)} + 1}\right)^m\|\mathbf{u}_\mu-0\|_{\mathbb{A}(\mu)},
\end{equation*} 
which completes the proof.
\end{proof}

To proceed, we consider the energy norm $\|v\|_{A(\mu)}:=\sqrt{a_\mu(v,v)}$ on $\mathcal{H}$. It follows from the boundedness and coercive property in \eqref{boundcoercive} that
\begin{equation}\label{energy}
\begin{aligned}
    &\|v\|_{A(\mu)}^2=a_\mu(v,v)\geq c_2\|v\|_{\mathcal{H}}^2,\quad v\in\mathcal{H},\\
    &\|u_\mu\|_{A(\mu)}:=\sup_{v\in\mathcal{H}}\frac{a_\mu(u_\mu,v)}{\|v\|_{A(\mu)}}\leq\frac{1}{\sqrt{c_2}}\sup_{v\in\mathcal{H}} \frac{|f(v)|}{\|v\|_{\mathcal{H}}}=\frac{1}{\sqrt{c_2}}\|f\|_{\mathcal{H}^*}.
\end{aligned}
\end{equation}
Note that $\|v\|_{A(\mu)}^2=(\mathbb{A}(\mu)\mathbf{v},\mathbf{v})$, where $\mathbf{v}$ is the coordinate vector of $v\in\mathcal{H}$.
Then using \eqref{energy} and Theorem \ref{errorRCGBMtwo}, we obtain
\begin{equation*}
    \|u_\mu-\hat{u}_\mu\|_{\mathcal{H}}\leq\frac{2}{c_2}\left(\frac{\sqrt{\kappa(\mu)} - 1}{\sqrt{\kappa(\mu)} + 1}\right)^m\|f\|_{\mathcal{H}^*}.
\end{equation*} 
Due to the assumption \eqref{boundcoercive}, it is straightforward to see  $\kappa(\mathbb{A}(\mu_0)^{-1}\mathbb{A}(\mu))\leq c_1/c_2$ for all $\mu\in\mathcal{P}$. Therefore, we obtain the following uniform error estimate:
\begin{equation}\label{RBMerror}
\sigma_m:=\sup_{\mu\in\mathcal{P}}\|u_\mu - \hat{u}_\mu\|_{\mathcal{H}}\leq \frac{2}{c_2}\left(\frac{c_1/c_2 - 1}{c_1/c_2 + 1}\right)^m\|f\|_{\mathcal{H}^*}. 
\end{equation}
Here $\sigma_m$ measures how well the reduced basis subspace $\mathcal{K}_m$ approximates the high-fidelity solution manifold $\{u_\mu\in\mathcal{H}: \mu\in\mathcal{P}\}.$

\subsubsection{Kolmogorov $n$-width}
The convergence analysis of the RCGBM also leads to a transparent analysis of the Kolmogorov $n$-width of the solution manifold of parametric PDEs. For a Sobolev space  $\tilde{\mathcal{H}}$ and $f\in\tilde{\mathcal{H}}^*$, we consider the continuous variational problem: for each $\mu\in\mathcal{P}$ find $\tilde{u}_\mu\in\tilde{\mathcal{H}}$ such that
\begin{equation}\label{twoparacts}
    \theta_1(\mu)a_1(\tilde{u}_\mu,v)+\theta_2(\mu)a_2(\tilde{u}_\mu,v)=f(v),\quad v\in\tilde{\mathcal{H}}.
\end{equation}
In practice, $\tilde{\mathcal{H}}$ could be $H_0^1(\Omega)$, $[H_0^1(\Omega)]^3$ or $H({\rm curl},\Omega)$, see examples in Subsections \ref{subsecelliptic}, \ref{subseclinearelasticity}, and \ref{subsecMaxwell}.
The set of all solutions or the solution manifold of \eqref{twoparacts} is 
\begin{equation*}
    \mathcal{M}=\big\{\tilde{u}_\mu\in\tilde{\mathcal{H}}: \mu\in\mathcal{P}\big\}.
\end{equation*}
For the ease of analysis, we then rewrite \eqref{twoparacts} as the following operator equation
\begin{equation}
    A(\mu)\tilde{u}_\mu=(\theta_1(\mu)A_1+\theta_2(\mu)A_2)u_\mu=f,
\end{equation}
where $A_1: \mathcal{H}\rightarrow\mathcal{H}^*$ and $A_2: \mathcal{H}\rightarrow\mathcal{H}^*$ are determined by
\begin{equation*}
    (A_jv)(w)=a_i(v,w),\quad v,w\in\mathcal{H}.
\end{equation*}
Classical error analysis of RBMs often relates the convergence rate of greedy-type RBMs to the Kolmogorov $m$-width of $\mathcal{M}$ (cf.~\cite{BinevCohenDahmenDeVore2011,DeVorePetrova2013,Wojtaszczyk2015,LiSiegel2024}):
\begin{equation*}
    d_m(\mathcal{M})_{\tilde{\mathcal{H}}}:=\inf_{V_m}\sup_{\tilde{u}_\mu\in\mathcal{M}}\inf_{v\in V_m}\|\tilde{u}_\mu-v\|_{\tilde{\mathcal{H}}},
\end{equation*}
where the first infimum is taken over all $m$-dimensional subspaces of $\tilde{\mathcal{H}}$. When $\mathcal{M}$ is compact, we have $d_m(\mathcal{M})_{\tilde{\mathcal{H}}}\to0$ as $m\to\infty$.  In approximation theory, $d_m(\mathcal{M})_{\tilde{\mathcal{H}}}$ is an important concept that
measures the best possible rate of approximation of $\mathcal{M}$ by $n$-dimensional subspaces.
\begin{corollary}
Let $\mathcal{M}$ be the solution manifold of \eqref{twoparacts}. Assume that 
    \begin{align*}
        a_\mu(v,w)&\leq c_1\|v\|_{\tilde{\mathcal{H}}}\|w\|_{\tilde{\mathcal{H}}},\\
        a_\mu(v,v)&\geq c_2\|v\|_{\tilde{\mathcal{H}}}^2
    \end{align*}
for all $v, w\in \tilde{\mathcal{H}}$. Then it holds that
\begin{equation*}
    d_m(\mathcal{M})_{\tilde{\mathcal{H}}}\leq \frac{2}{c_2}\left(\frac{c_1/c_2-1}{c_1/c_2+1}\right)^m\|f\|_{\tilde{\mathcal{H}}^*}.
\end{equation*}
\end{corollary}
\begin{proof}
Let $B=A(\mu_0)^{-1}$ and some fixed $\mu_1$ we set
\begin{equation*}
    \tilde{\mathcal{K}}_m:={\rm Span}\big\{Bf,BA(\mu_1)Bf, \ldots,(BA(\mu_1))^{m-1}Bf\big\}.
\end{equation*}
By redoing the same analysis of \eqref{RBMerror} in the Hilbert space $\tilde{\mathcal{H}}$, we have
\begin{equation*}
        \|\tilde{u}_\mu-P_{\tilde{\mathcal{K}}_m}\tilde{u}_\mu\|_{\tilde{\mathcal{H}}}\leq\frac{2}{c_2}\left(\frac{c_1/c_2-1}{c_1/c_2+1}\right)^m\|f\|_{\tilde{\mathcal{H}}^*}.
\end{equation*}
Combining it with $d_m(\mathcal{M})_{\tilde{\mathcal{H}}}\leq\|\tilde{u}_\mu-P_{\mathcal{K}_m}\tilde{u}_\mu\|_{\tilde{\mathcal{H}}}$ completes the proof.
\end{proof}

\section{Nonsymmetric and Indefinite Parametric PDEs}\label{secNSPD} In this section, we extend the RCGBM to nonsymmetric and indefinite parametric PDEs:
\begin{equation}\label{multiparabilinear}
a_\mu(u_\mu,v)=\theta_1(\mu)a_1(u_\mu,v)+\cdots+\theta_J(\mu)a_J(u_\mu,v)=f(v),\quad v\in\mathcal{H},
\end{equation}
where $a_\mu$ is not necessarily an SPD bilinear form. Following the same convention as in \eqref{Amuf}, we have  the equivalent matrix-vector equation
\begin{equation}\label{multiparamatrix}
\mathbb{A}(\mu)\mathbf{u}_\mu=\big(\theta_1(\mu)\mathbb{A}_1+\cdots+\theta_J(\mu)\mathbb{A}_J\big)\mathbf{u}_\mu=\mathbf{f}.
\end{equation}
Here $\mathbb{A}(\mu)$ is not necessarily an SPD matrix. To ensure the well-posedness of \eqref{multiparabilinear} and \eqref{multiparamatrix}, we assume that $a_\mu$ satisfies
\begin{subequations}
    \begin{align}
        &\sup_{0\neq v\in\mathcal{H}}\sup_{0\neq v\in\mathcal{H}}\frac{a_\mu(v,w)}{\|v\|_{\mathcal{H}}\|w\|_{\mathcal{H}}}\leq c_1,\\
        &\inf_{0\neq v\in\mathcal{H}}\sup_{0\neq v\in\mathcal{H}} \frac{a_\mu(v,w)}{\|v\|_{\mathcal{H}}\|w\|_{\mathcal{H}}}\geq c_2>0.
    \end{align}
\end{subequations}
In this case, the PCG and RCGBM are not applicable anymore. A natural remedy is to construct the reduced basis subspace by more general iterative solvers such as the Generalized Minimum Residual (GMRES), see, e.g., \cite{Saad2003}. The corresponding RBM is named as ``Reduced Krylov Basis Method of the first kind" (${\rm RKBM}_1$), see Algorithm \ref{a:RKBM1} for details.

\begin{algorithm}[thp]
  \caption{${\rm RKBM_1}(\mathbb{B}, \mathbb{M}, \mathcal{P}, \mu_0, \mu_1,  m)$ for solving $\mathbb{A}(\mu)\mathbf{u}_\mu=\mathbf{f}$}\label{a:RKBM1}
  \begin{algorithmic}
  \State \textbf{Input:}  a parameter set $\mathcal{P}$ of practical interest, $\mu_0\neq\mu_1\in\mathcal{P}$, 
 an integer $m>0$, a preconditioner $\mathbb{B}: \mathbb{R}^n\rightarrow\mathbb{R}^n$ for $\mathbb{A}(\mu_0)$, and an SPD matrix $\mathbb{M}$.
    \State\textbf{Offline Stage:} 
    use the GMRES for $\mathbb{B}\mathbb{A}(\mu_1)\mathbf{u}_{\mu_1}=\mathbb{B}\mathbf{f}$ with initial guess $\mathbf{u}_0=\mathbf{0}$ to generate iterative solutions $\mathbf{u}_1, \ldots, \mathbf{u}_{m-1}$ and $\mathbf{z}_j=\mathbb{B}(\mathbf{f}-\mathbb{A}(\mu_1)\mathbf{u}_j)$, $j=0, \ldots, m-1$; 
    \State construct the reduced basis subspace \[
   \mathcal{K}_m={\rm Span}\big\{\mathbf{z}_0, \mathbf{z}_1, \ldots, \mathbf{z}_{m-1}\big\}
   \]
   and the projection matrix 
   \begin{align*}
   \mathbb{P}=\left[\frac{\mathbf{z}_0}{\|\mathbf{z}_0\|},\frac{\mathbf{z}_1}{\|\mathbf{z}_1\|},\dots,\frac{\mathbf{z}_{m-1}}{\|\mathbf{z}_{m-1}\|}\right].
   \end{align*}

   \State \textbf{Online Evaluation:} for each $\mu\in\mathcal{P}$, compute the reduced basis solution $\hat{\mathbf{u}}_\mu\in\mathcal{K}_m$ such that 
\begin{equation*}
\|\mathbb{B}\mathbf{f}-\mathbb{B}\mathbb{A}(\mu)\hat{\mathbf{u}}_\mu\|_{\mathbb{M}}=\min_{\mathbf{v}\in\mathcal{K}_m}\|\mathbb{B}\mathbf{f}-\mathbb{B}\mathbb{A}(\mu)\mathbf{v}\|_{\mathbb{M}}.
\end{equation*} 
Equivalently, $\hat{\mathbf{u}}_\mu=\mathbb{P}((\mathbb{B}\mathbb{A}(\mu)\mathbb{P})^t\mathbb{M}\mathbb{B}\mathbb{A}(\mu)\mathbb{P})^{-1}(\mathbb{B}\mathbb{A}(\mu)\mathbb{P})^t\mathbb{M}\mathbb{B}\mathbf{f}$.
  \end{algorithmic}  
\end{algorithm}  

In the above algorithm, we determine the online solution $\hat{\mathbf{u}}_\mu$ by a small-scale least-squares problem with respect to the norm $(\bullet,\bullet)_{\mathbb{M}}$ in the Krylov subspace $\mathcal{K}_m$. When $(\bullet,\bullet)_{\mathbb{M}}=(\bullet,\bullet)$ is the $\ell_2$-norm, $\hat{\mathbf{u}}_\mu$ coincides with the iterative solution of a standard GMRES method for $\mathbb{B}\mathbb{A}(\mu)\mathbf{u}_\mu=\mathbb{B}\mathbf{f}$ at the $m$-th step.

To efficiently evaluate the online solution $\hat{\mathbf{u}}_\mu$ in Algorithm \ref{a:RKBM1}, it is important to pre-compute and store the products 
 $(\mathbb{B}\mathbb{A}_j\mathbb{P})^t\mathbb{M}\mathbb{B}\mathbf{f}$ and    $(\mathbb{B}\mathbb{A}_j\mathbb{P})^t\mathbb{M}(\mathbb{B}\mathbb{A}_k\mathbb{P})$ with $1\leq j, k\leq J$. Then for each paramter $\mu\in\mathcal{P}$, one could assemble
\begin{align*}
    &(\mathbb{B}\mathbb{A}(\mu)\mathbb{P})^t\mathbb{M}(\mathbb{B}\mathbb{A}(\mu)\mathbb{P})=\sum_{j, k=1}^J\theta_j(\mu)\theta_k(\mu)(\mathbb{B}\mathbb{A}_j\mathbb{P})^t\mathbb{M}(\mathbb{B}\mathbb{A}_k\mathbb{P})\\
    &(\mathbb{B}\mathbb{A}(\mu)\mathbb{P})^t\mathbb{M}\mathbb{B}\mathbf{f}=\sum_{j=1}^J\theta_j(\mu)(\mathbb{B}\mathbb{A}_j\mathbb{P})^t\mathbb{M}\mathbb{B}\mathbf{f}
\end{align*}
and complete the inversion process for $\hat{\mathbf{u}}_\mu$
without $O(n)$ computational cost. 

Another popular family of general-purpose iterative solvers is the Bi-Conjugate Gradient (BiCG) method and its variants (cf.~\cite{VanDerVorst1992,BankChan1993,Saad2003}). The corresponding RBM is termed ``Reduced Krylov Basis Method of the second kind" (${\rm RKBM}_2$), see Algorithm \ref{a:RKBM2} for details. In particular, the reduced basis subspace in ${\rm RKBM}_2$ is generated by BiCG iterations and the online output $\hat{\mathbf{u}}_\mu$ of ${\rm RKBM}_2$ is the same as  the $m$-th iterate of a BiCG method for $\mathbb{B}\mathbb{A}(\mu)\mathbf{u}_\mu=\mathbb{B}\mathbf{f}$.

\begin{algorithm}[thp]
  \caption{${\rm RKBM}_2(\mathbb{B}, \mathcal{P}, \mu_0, \mu_1,  m)$ for solving $\mathbb{A}(\mu)\mathbf{u}_\mu=\mathbf{f}$ in $\mathbb{R}^n$}\label{a:RKBM2}
  \begin{algorithmic}
    \State \textbf{Input:}  a parameter set $\mathcal{P}$ of practical interest, $\mu_0\neq\mu_1\in\mathcal{P}$, an integer $m>0$, a preconditioner $\mathbb{B}: \mathbb{R}^n\rightarrow\mathbb{R}^n$ for $\mathbb{A}(\mu_0)$ and its adjoint $\mathbb{B}^t$.
    \State\textbf{Offline Stage:} set 
    $\mathbf{u}_0=\mathbf{0}$,
    $\mathbf{r}_0=\mathbf{f}$, $\mathbf{p}_0=\mathbf{z}_0=\mathbb{B}\mathbf{r}_0$;

    \State 
    select some $\mathbf{r}_0^*$ with $(\mathbf{r}_0,\mathbf{r}_0^*)\neq0$ and set $\mathbf{p}^*_0=\mathbf{z}^*_0=\mathbb{B}^t(\mathbf{r}^*_0)$;
   \For{$k=1:(m-1)$}
   \State $\alpha_k=(\mathbf{r}_{k-1},\mathbf{z}^*_{k-1})/(\mathbb{A}(\mu_1)\mathbf{p}_{k-1},\mathbf{p}^*_{k-1})$;
   \State $\mathbf{u}_k =  \mathbf{u}_{k-1} + \alpha_k\mathbf{p}_{k-1}$; 
   \State $\mathbf{r}_k =  \mathbf{r}_{k-1} - \alpha_k\mathbb{A}(\mu_1)\mathbf{p}_{k-1}$;  
   \State $\mathbf{r}^*_k =  \mathbf{r}^*_{k-1} - \alpha_k\mathbb{A}({\mu_1})^t\mathbf{p}^*_{k-1}$;
   \State $\mathbf{z}_{k} = \mathbb{B}\mathbf{r}_{k}$;
   \State $\mathbf{z}^*_{k} = \mathbb{B}^t(\mathbf{r}^*_{k})$;
   \State $\beta_k = (\mathbf{r}_{k},\mathbf{z}_{k}^*)/(\mathbf{r}_{k-1},\mathbf{z}_{k-1}^*)$;
   \State $\mathbf{p}_{k} = \mathbf{z}_{k} + \beta_k \mathbf{p}_{k-1}$; 
   \State $\mathbf{p}^*_{k} = \mathbf{z}^*_{k} + \beta_k \mathbf{p}^*_{k-1}$;
   \EndFor
   \State construct the projection matrices \begin{align*}
       \mathbb{P}&:=\left[\frac{\mathbf{p}_0}{\|\mathbf{p}_0\|},\frac{\mathbf{p}_1}{\|\mathbf{p}_1\|},\dots,\frac{\mathbf{p}_{m-1}}{\|\mathbf{p}_{m-1}\|}\right],\\
       \mathbb{Q}&:=\left[\frac{\mathbf{p}^*_0}{\|\mathbf{p}^*_0\|},\frac{\mathbf{p}^*_1}{\|\mathbf{p}^*_1\|},\dots,\frac{\mathbf{p}^*_{m-1}}{\|\mathbf{p}^*_{m-1}\|}\right];
   \end{align*}
   set the reduced basis spaces as $\mathcal{K}_m={\rm col}(\mathbb{P})$, $\mathcal{L}_m={\rm col}(\mathbb{Q})$.
   \State \textbf{Online Evaluation:} for each $\mu\in\mathcal{P}$, find the reduced basis solution $\hat{\mathbf{u}}_\mu\in\mathcal{K}_m$ such that 
\begin{equation*}
(\mathbb{A}(\mu)\hat{\mathbf{u}}_\mu,\mathbf{v})=(\mathbf{f},\mathbf{v}),\quad \mathbf{v}\in\mathcal{L}_m.
\end{equation*} 
Equivalently, $\hat{\mathbf{u}}_\mu=\mathbb{P}(\mathbb{Q}^\top\mathbb{A}(\mu)\mathbb{P})^{-1}\mathbb{Q}^\top\mathbf{f}$.
  \end{algorithmic}  
\end{algorithm}

\subsection{Linear PDEs with Two Parameters}
Although the algorithms ${\rm RKBM}_1$ and ${\rm RKBM}_2$ are applicable to general parametric PDEs, we are not able to derive  a concise and transparent error analysis in the most general case. Therefore, we shall focus on the two-parameter variational problem:
\begin{equation}\label{twoparaNSPD}
a_\mu(u_\mu,v)=\theta_1(\mu)a_1(u_\mu,v)+\theta_2(\mu)a_2(u_\mu,v)=f(v),\quad v\in\mathcal{H},
\end{equation}
or the equivalent two-parameter matrix-vector problem:
\begin{equation}\label{twoparaNSPDmatrix}
\mathbb{A}(\mu)\mathbf{u}_\mu=(\theta_1(\mu)\mathbb{A}_1+\theta_2(\mu)\mathbb{A}_2)\mathbf{u}_\mu=\mathbf{f}.
\end{equation}
Several model problems of the form \eqref{twoparaNSPD} or \eqref{twoparaNSPDmatrix} are described below. 

\subsubsection{Convection-Diffusion Equation}\label{subsecConvectionDiffusion}
Let $\bm{b}\in C^1(\overline{\Omega})$ be a divergence-free vector field and $f\in L^2(\Omega)$. Let $\mu=(\nu_1,\nu_2)$ and $\mathcal{P}\subset \mathbb{R}_+\times[0,2\pi]$ be a compact subset. For each $\mu\in\mathcal{P}$ we consider the convection-diffusion equation
\begin{equation*}
\begin{aligned}
-\nu_1\Delta\tilde{u}_\mu + \cos(\nu_2)\bm{b}\cdot\nabla \tilde{u}_\mu&=f\quad\text{ in }\Omega,\\
\tilde{u}_\mu&=0\quad\text{ on }\partial\Omega.
\end{aligned}
\end{equation*}
The corresponding FEM based on a finite element space $\mathcal{H}\subset H_0^1(\Omega)$ is of the form \eqref{twoparaNSPD},
where $\theta_1(\mu)=\nu_1$, $\theta_2(\mu)=\cos(\nu_2)$, and $a_1, a_2$ are given by 
\begin{align*}
a_1(v,w)&=\int_\Omega\nabla v\cdot\nabla wdx,\\
    a_2(v,w)&=\int_\Omega (\bm{b}\cdot\nabla v)wdx.
\end{align*}

\subsubsection{Stokes Equation}
The second example is the Stokes equation 
\begin{equation*}
\begin{aligned}
      -\mu\Delta\tilde{\bm{u}}_\mu+\nabla\tilde{p}_\mu&=\bm{f}\quad\text{ in }\Omega,\\
      \nabla\cdot\tilde{\bm{u}}_\mu&=0\quad\text{ in }\Omega,
\end{aligned}
\end{equation*}
subject to the no-slip boundary condition $\tilde{\bm{u}}_\mu|_{\partial\Omega}=0$. Let $L_0^2(\Omega)=\{q\in L^2(\Omega): \int_\Omega qdx=0\}$. Let $\mathcal{H}=V\times Q\subset[H_0^1(\Omega)]^3\times L_0^2(\Omega)$ be an inf-sup stable Stokes finite element space pair, e.g., the Taylor--Hood element (cf.~\cite{GiraultRaviart1986}). For each $\mu>0$ the FEM seeks $(\bm{u}_\mu,p_\mu)\in\mathcal{H}$ such that 
\begin{equation*}
    \theta_1(\mu)a_1(\bm{u}_\mu,p_\mu;\bm{v},q)+\theta_2(\mu)a_2(\bm{u}_\mu,p_\mu;\bm{v},q)=\int_\Omega \bm{f}\cdot\bm{v}dx,\quad (\bm{v},q)\in\mathcal{H},
\end{equation*}
where $\theta_1(\mu)=\mu$, $\theta_2(\mu)=1$, and 
\begin{align*}
    &a_1(\bm{v},q;\bm{w},r)=\int_\Omega\nabla \bm{v}\cdot\nabla \bm{w}dx,\\
    &a_2(\bm{v},q;\bm{w},r)=-\int_\Omega(\nabla\cdot \bm{v})r dx-\int_\Omega(\nabla\cdot\bm{w})q dx.
\end{align*}

\subsubsection{Helmholtz Equation}
The third example is the indefinite problem 
\begin{align*}
    -\Delta\tilde{u}_\mu-\mu^2\tilde{u}_\mu&=f\quad\text{ in }\Omega,\\
    \tilde{u}_\mu&=0\quad\text{ on }\partial\Omega,
\end{align*}
which is a Helmholtz-type equation. Let $\mathcal{H}\subset H_0^1(\Omega)$ be the Lagrange finite element space. The corresponding FEM seeks $u_\mu\in\mathcal{H}$ satisfying
\begin{equation*}
    \int_\Omega\nabla u_\mu\cdot\nabla vdx-\mu^2\int_\Omega u_\mu vdx=\int_\Omega fvdx,\quad v\in\mathcal{H}.
\end{equation*}
The above FEM is of the form \eqref{twoparaNSPD} with $\theta_1(\mu)=1$,  $\theta_2(\mu)=-\mu^2$,  and 
\begin{align*}
    a_1(v,w)&=\int_\Omega\nabla v\cdot\nabla wdx,\\
    a_2(v,w)&=\int_\Omega vwdx.
\end{align*}

\subsection{Convergence of GMRES-type Reduced Basis}

In this subsection, we present a transparent convergence analysis of ${\rm RKBM}_1$ using the Field-of-Value (FOV) analysis of the GMRES (cf.~\cite{Elman1982,Ma2016}).
\begin{theorem}\label{errorRKBM1}
Let $\mathbb{A}(\mu)=\theta_1(\mu)\mathbb{A}_1+\theta_2(\mu)\mathbb{A}_2$, $\mathbb{B}=\mathbb{A}(\mu_0)^{-1}$, $\kappa(\mu)=\kappa(\mathbb{B}\mathbb{A}(\mu))$, 
\begin{align*}
\gamma(\mu):=\inf_{\mathbf{v}\neq\mathbf{0}}\frac{(\mathbb{B}\mathbb{A}(\mu)\mathbf{v},\mathbf{v})_{\mathbb{M}}}{(\mathbf{v},\mathbf{v})_{\mathbb{M}}},\quad\Gamma(\mu):=\sup_{\mathbf{w}\neq\mathbf{0}}\frac{\|\mathbb{B}\mathbb{A}(\mu)\mathbf{w}\|_{\mathbb{M}}}{\|\mathbf{w}\|_{\mathbb{M}}}.
\end{align*}
Assume $\theta_1(\mu_0)\theta_2(\mu_1)-\theta_2(\mu_0)\theta_1(\mu_1)\neq0$. Then for any $\mu\in\mathcal{P}$ with $\theta_1(\mu_0)\theta_2(\mu)-\theta_2(\mu_0)\theta_1(\mu)\neq0$, the reduced basis solution $\hat{\mathbf{u}}_\mu$ of ${\rm RKBM}_1(\mathbb{B}, \mathbb{M}, \mathcal{P}, \mu_0, \mu_1, m)$ satisfies
\begin{equation*}
      \|\mathbb{B}\mathbb{A}(\mu)(\mathbf{u}_\mu-\hat{\mathbf{u}}_\mu)\|_{\mathbb{M}}\leq\left(1-\frac{\gamma(\mu)^2}{\Gamma(\mu)^2}\right)^{\frac{m}{2}}\|\mathbb{B}\mathbf{f}\|_{\mathbb{M}}.
  \end{equation*}
\end{theorem}
\begin{proof}
By the property of the GMRES method, the reduced basis subspace $\mathcal{K}_m$ in ${\rm RKBM}_1(\mathbb{B}, \mathbb{M}, \mathcal{P}, \mu_0, \mu_1, m)$ is indeed a Krylov subspace:
\[
   \mathcal{K}_m=\mathcal{K}_m(\mathbb{B}\mathbb{A}(\mu_1),\mathbb{B}\mathbf{f})={\rm Span}\big\{\mathbb{B}\mathbf{f}, \mathbb{B}\mathbb{A}(\mu_1)\mathbb{B}\mathbf{f}, \ldots, (\mathbb{B}\mathbb{A}(\mu_1))^m\mathbb{B}\mathbf{f}\big\}.
   \]
Lemma \ref{invariancelemma} implies that $\mathcal{K}_m(\mathbb{B}\mathbb{A}(\mu_1),\mathbb{B}\mathbf{f})$ is the same as $\mathcal{K}_m(\mathbb{B}\mathbb{A}(\mu),\mathbb{B}\mathbf{f})$, i.e., the Krylov subspace produced by the GMRES iteration for  $\mathbb{B}\mathbb{A}(\mu)\mathbf{u}_\mu=\mathbb{B}\mathbf{f}$. 

As a consequence, the accuracy of $\hat{\mathbf{u}}_\mu$ could be analyzed within the FOV framework. To indicate the dependence of $\hat{\mathbf{u}}_\mu$ on $m$, we set $\hat{\mathbf{u}}^m_{\mu}:=\hat{\mathbf{u}}_\mu$. In addition, we make use of  the residual $\mathbf{r}_m:=\mathbf{f}-\mathbb{A}(\mu)\hat{\mathbf{u}}^m_\mu$, the subspace $\mathcal{W}_m:=\mathbb{B}\mathbb{A}(\mu)\mathcal{K}_m$, and the orthogonal projection $P_{\mathcal{W}_m}$ onto $\mathcal{W}_m$ with respect to $(\bullet,\bullet)_{\mathbb{M}}$. The least-squares property of $\hat{\mathbf{u}}^m_{\mu}$ implies that 
\[\mathbb{B}\mathbb{A}(\mu)\hat{\mathbf{u}}_\mu=P_{\mathcal{W}_m}\mathbb{B}\mathbf{f}.\]
Let $I$ be the identity operator. Using the above identity and $I-P_{\mathcal{W}_m}=P_{\mathcal{W}^\perp_m}$ with $\mathcal{W}^\perp_{m}\subseteq\mathcal{W}^\perp_{m-1}$, we obtain
\begin{align*}
    \mathbb{B}\mathbf{r}_m&=(I-P_{\mathcal{W}_m})\mathbb{B}\mathbf{f}=P_{\mathcal{W}_m^\perp}\mathbb{B}\mathbf{f}\\
    &=P_{\mathcal{W}_m^\perp}P_{\mathcal{W}_{m-1}^\perp}\mathbb{B}\mathbf{f}=P_{\mathcal{W}_m^\perp}\mathbb{B}\mathbf{r}_{m-1}\\
    &=(I-P_{\mathcal{W}_m})\mathbb{B}\mathbf{r}_{m-1}.
\end{align*} 
In what follows,
\begin{equation}\label{recurrence}
    \begin{aligned}
        \|\mathbb{B}\mathbf{r}_m\|_{\mathbb{M}}^2&=\|\mathbb{B}\mathbf{r}_{m-1}\|_{\mathbb{M}}^2-\|P_{\mathcal{W}_m}\mathbb{B}\mathbf{r}_{m-1}\|_{\mathbb{M}}^2\\
        &=\|\mathbb{B}\mathbf{r}_{m-1}\|_{\mathbb{M}}^2-\Big(\sup_{\mathbf{0}\neq\mathbf{w}\in\mathcal{W}_m}\frac{(\mathbb{B}\mathbf{r}_{m-1},\mathbf{w})_{\mathbb{M}}}{\|\mathbf{w}\|_{\mathbb{M}}\|\mathbb{B}\mathbf{r}_{m-1}\|_{\mathbb{M}}}\Big)^2\|\mathbb{B}\mathbf{r}_{m-1}\|_{\mathbb{M}}^2.
    \end{aligned}
\end{equation}
To the estimate the supremum in \eqref{recurrence}, we proceed as follows:
\begin{equation}\label{supestimate}
    \begin{aligned}
&\sup_{\mathbf{0}\neq\mathbf{w}\in\mathcal{W}_m}\frac{|(\mathbb{B}\mathbf{r}_{m-1},\mathbf{w})_{\mathbb{M}}|}{\|\mathbf{w}\|_{\mathbb{M}}\|\mathbb{B}\mathbf{r}_{m-1}\|_{\mathbb{M}}}\geq\inf_{\mathbf{0}\neq\mathbf{v}\in\mathcal{K}_m}\sup_{\mathbf{0}\neq\mathbf{w}\in\mathcal{W}_m}\frac{|(\mathbf{v},\mathbf{w})_{\mathbb{M}}|}{\|\mathbf{w}\|_{\mathbb{M}}\|\mathbf{v}\|_{\mathbb{M}}}\\
&=\inf_{\mathbf{0}\neq\mathbf{v}\in\mathcal{K}_m}\sup_{\mathbf{0}\neq\mathbf{w}\in\mathcal{K}_m}\frac{|(\mathbf{v},\mathbb{B}\mathbb{A}(\mu)\mathbf{w})_{\mathbb{M}}|}{\|\mathbb{B}\mathbb{A}(\mu)\mathbf{w}\|_{\mathbb{M}}\|\mathbf{v}\|_{\mathbb{M}}}\geq\inf_{\mathbf{0}\neq\mathbf{v}\in\mathcal{K}_m}\frac{|(\mathbf{v},\mathbb{B}\mathbb{A}(\mu)\mathbf{v})_{\mathbb{M}}|}{\|\mathbb{B}\mathbb{A}(\mu)\mathbf{v}\|_{\mathbb{M}}\|\mathbf{v}\|_{\mathbb{M}}}\\
&\geq\inf_{\mathbf{0}\neq\mathbf{v}\in\mathbb{R}^n}\frac{|(\mathbb{B}\mathbb{A}(\mu)\mathbf{v},\mathbf{v})_{\mathbb{M}}|}{\|\mathbb{B}\mathbb{A}(\mu)\mathbf{v}\|_{\mathbb{M}}\|\mathbf{v}\|_{\mathbb{M}}}.
    \end{aligned}
\end{equation}
Using \eqref{supestimate} and the  simple splitting 
\[
\frac{|(\mathbb{B}\mathbb{A}(\mu)\mathbf{v},\mathbf{v})_{\mathbb{M}}|}{\|\mathbb{B}\mathbb{A}(\mu)\mathbf{v}\|_{\mathbb{M}}\|\mathbf{v}\|_{\mathbb{M}}}=\frac{|(\mathbb{B}\mathbb{A}(\mu)\mathbf{v},\mathbf{v})_{\mathbb{M}}|}{(\mathbf{v},\mathbf{v})_{\mathbb{M}}}\frac{\sqrt{(\mathbf{v},\mathbf{v})_{\mathbb{M}}}}{\|\mathbb{B}\mathbb{A}(\mu)\mathbf{v}\|_{\mathbb{M}}},
\]
we obtain the lower bound
\[
\sup_{\mathbf{0}\neq\mathbf{w}\in\mathcal{W}_m}\frac{|(\mathbb{B}\mathbf{r}_{m-1},\mathbf{w})_{\mathbb{M}}|}{\|\mathbf{w}\|_{\mathbb{M}}\|\mathbb{B}\mathbf{r}_{m-1}\|_{\mathbb{M}}}\geq\frac{\gamma(\mu)}{\Gamma(\mu)}.
\]
Combining it with \eqref{recurrence} leads to 
\[
\|\mathbb{B}\mathbf{r}_m\|_{\mathbb{M}}^2\leq\left(1-\frac{\gamma(\mu)^2}{\Gamma(\mu)^2}\right)\|\mathbb{B}\mathbf{r}_{m-1}\|_{\mathbb{M}}^2.
\]
The proof is complete.
\end{proof}

To illustrate the result in Theorem \ref{errorRKBM1}, we use ${\rm RKBM}_1(\mathbb{B}, \mathbb{M},\mathcal{P}, 0,\mu_1,m)$ with $\mathbb{B}=\mathbb{A}_1^{-1}$, $\mathbb{M}=\mathbb{A}_1$ to solve the parametric convection-diffusion problem in Subsection \ref{subsecConvectionDiffusion}. Recall that $\{\phi_i\}_{1\leq i\leq n}$ is a basis of $\mathcal{H}$. For $\mathbf{v}=(v_1,\ldots,v_n)^t$, let $v=(\phi_1,\ldots,\phi_n)\mathbf{v}$ be the finite element function in $\mathcal{H}$. Assume that $\mu_1$ is contained in some positive interval  $[\underline{\mu},\overline{\mu}]$. Then we have
\begin{equation}\label{gamma}
    \begin{aligned}
        (\mathbb{B}\mathbb{A}(\mu)\mathbf{v},\mathbf{v})_{\mathbb{M}}&=(\mathbb{A}(\mu)\mathbf{v},\mathbf{v})=\int_\Omega\mu_1\nabla v\cdot\nabla vdx+\cos(\mu_2)\int_\Omega(\bm{b}\cdot\nabla v)vdx\\
        &=\int_\Omega\mu_1\nabla v\cdot\nabla vdx\geq \underline{\mu}\|\mathbf{v}\|_{\mathbb{M}}^2,
    \end{aligned}
\end{equation}
where the third identity above relies on that $\bm{b}$ is divergence-free. On the other hand, let $c_P$ be the Poincar\'e constant, i.e., $\|w\|_{L^2(\Omega)}\leq c_P|w|_{H^1(\Omega)}$ for any $w\in\mathcal{H}$. Direct calculation shows that
\begin{equation}\label{Gamma}
    \begin{aligned}
        &\|\mathbb{B}\mathbb{A}(\mu)\mathbf{v}\|_{\mathbb{M}}=\sup_{\mathbf{0}\neq\mathbf{w}\in\mathbb{R}^n}\frac{(\mathbb{B}\mathbb{A}(\mu)\mathbf{v},\mathbf{w})_{\mathbb{M}}}{\|\mathbf{w}\|_{\mathbb{M}}}=\sup_{\mathbf{0}\neq\mathbf{w}\in\mathbb{R}^n}\frac{(\mathbb{A}(\mu)\mathbf{v},\mathbf{w})}{\|\mathbf{w}\|_{\mathbb{M}}}\\
        &=\sup_{w\in\mathcal{H}, |w|_{H^1(\Omega)}=1}\left(\mu_1\int_\Omega\nabla v\cdot\nabla wdx+\cos(\mu_2)\int_\Omega (\bm{b}\cdot\nabla v) wdx\right)\\
        &\leq\overline{\mu}|v|_{H^1(\Omega)}+c_P\|\bm{b}\|_{L^\infty(\Omega)}|v|_{H^1(\Omega)} =(\bar{\mu}+c_P\|\bm{b}\|_{L^\infty(\Omega)})\|\mathbf{v}\|_{\mathbb{M}}.
    \end{aligned}
\end{equation}
It then follows from that \eqref{gamma} and \eqref{Gamma} that 
\begin{align*}
    \gamma(\mu)&\geq\underline{\mu},\\
    \Gamma(\mu)&\leq\bar{\mu}+c_P\|\bm{b}\|_{L^\infty(\Omega)}.
\end{align*}
Combining Theorem \ref{errorRKBM1} with those  bounds of $\gamma(\mu)$, $\Gamma(\mu)$, we obtain an error estimate of the reduced basis solution $\hat{\mathbf{u}}_\mu$ from ${\rm RKBM}_1(\mathbb{B}, \mathbb{M}, \mathcal{P}, 0,\mu_1,m)$ for the convection-diffusion problem in Subsection \ref{subsecConvectionDiffusion}:
\begin{equation*}
      \|\mathbb{B}\mathbb{A}(\mu)(\mathbf{u}_\mu-\hat{\mathbf{u}}_\mu)\|_{\mathbb{M}}\leq\left(1-\frac{\underline{\mu}^2}{(\bar{\mu}+c_P\|\bm{b}\|_{L^\infty(\Omega)})^2}\right)^{\frac{m}{2}}\|\mathbb{B}\mathbf{f}\|_{\mathbb{M}}.
  \end{equation*}

\subsection{Convergence of BiCG-type Reduced Basis}
In comparison to the GMRES,
the convergence theory of BiCG-type iterative solvers is not well-developed. Among existing literature, the work \cite{BankChan1993} established a convergence estimate of BiCG-type methods using Babuska's inf-sup theory. Following the idea in \cite{BankChan1993}, we present a best approximation error estimate of the BiCG-type RBM in the next theorem.

\begin{theorem}\label{errorRKBM2}
Let $\mathbb{A}(\mu)=\theta_1(\mu)\mathbb{A}_1+\theta_2(\mu)\mathbb{A}_2$, $\mathbb{B}=\mathbb{A}(\mu_0)^{-1}$, $\kappa(\mu)=\kappa(\mathbb{B}\mathbb{A}(\mu))$. Let $\mathcal{K}_m^{\mu}=\mathcal{K}_m(\mathbb{B}\mathbb{A}(\mu),\mathbb{B}\mathbf{f})$ and $\mathcal{L}_m^{\mu}=\mathcal{K}_m(\mathbb{B}^t\mathbb{A}(\mu)^t,\mathbb{B}^t\mathbf{r}_0^*)$. Assume that there exist SPD matrices $\mathbb{M}, \mathbb{M}^\prime$ and constants $\alpha(\mu), \beta(\mu)>0$ such that 
\begin{subequations}\label{infsup}
\begin{align}
    &(\mathbb{A}(\mu)\mathbf{v},\mathbf{w})\leq\alpha(\mu)\|\mathbf{v}\|_{\mathbb{M}}\|\mathbf{w}\|_{\mathbb{M}^\prime},\\
    &\inf_{\mathbf{0}\neq\mathbf{v}\in\mathcal{K}_m^\mu}\sup_{\mathbf{0}\neq\mathbf{w}\in\mathcal{L}^\mu_m}\frac{(\mathbb{A}(\mu)\mathbf{v},\mathbf{w})}{\|\mathbf{v}\|_{\mathbb{M}}\|\mathbf{w}\|_{\mathbb{M}^\prime}}\geq\beta(\mu)
\end{align}
\end{subequations}
for all $\mathbf{v}, \mathbf{w}\in\mathbb{R}^n$. Assume $\theta_1(\mu_0)\theta_2(\mu_1)-\theta_2(\mu_0)\theta_1(\mu_1)\neq0$. 
Then for any $\mu\in\mathcal{P}$ with $\theta_1(\mu_0)\theta_2(\mu)-\theta_2(\mu_0)\theta_1(\mu)\neq0$, the solution $\hat{\mathbf{u}}_\mu$ of ${\rm RKBM}_2(\mathbb{B}, \mathcal{P}, \mu_0, \mu_1, m)$ satisfies
\begin{equation}\label{quasioptimal}
      \|\mathbf{u}_\mu - \hat{\mathbf{u}}_\mu\|_{\mathbb{M}}\leq \frac{\alpha(\mu)}{\beta(\mu)}\inf_{\mathbf{v}\in\mathcal{K}_m^{\mu}}\|\mathbf{u}_{\mu}-\mathbf{v}\|_{\mathbb{M}}.
  \end{equation}
\end{theorem}
\begin{proof}
It follows from the property of the BiCG method that $\mathcal{K}_m$ and  $\mathcal{L}_m$ in Algorithm \ref{a:RKBM1} are 
\begin{align*}
    &\mathcal{K}_m=\mathcal{K}_m(\mathbb{B}\mathbb{A}(\mu_1),\mathbb{B}\mathbf{f})={\rm Span}\big\{\mathbb{B}\mathbf{f},\mathbb{B}\mathbb{A}(\mu_1)\mathbb{B}\mathbf{f},\ldots, (\mathbb{B}\mathbb{A}(\mu_1))^m\mathbb{B}\mathbf{f}\big\},\\
    &\mathcal{L}_m=\mathcal{K}_m(\mathbb{B}^t\mathbb{A}(\mu_1)^t,\mathbb{B}\mathbf{f})={\rm Span}\big\{\mathbb{B}\mathbf{r}_0^*,\mathbb{B}^t\mathbb{A}(\mu_1)^t\mathbb{B}\mathbf{r}_0^*,\ldots, (\mathbb{B}^t\mathbb{A}(\mu_1)^t)^m\mathbb{B}\mathbf{r}_0^*\big\}.
\end{align*}
Then using Lemma \ref{invariancelemma}, we conclude that 
\begin{align*}
    \mathcal{K}_m&=\mathcal{K}_m(\mathbb{B}\mathbb{A}(\mu),\mathbb{B}\mathbf{f}),\\
    \mathcal{L}_m&=\mathcal{K}_m(\mathbb{B}^t\mathbb{A}(\mu)^t,\mathbb{B}\mathbf{f}),
\end{align*}
i.e., $\mathcal{K}_m$ and $\mathcal{L}_m$ are the two Krylov subspaces produced by  the BiCG method for $\mathbb{B}\mathbb{A}(\mu)\mathbf{u}_\mu=\mathbb{B}\mathbf{f}$ within $m$ iterations.
Recall that $\hat{\mathbf{u}}_\mu\in\mathcal{K}^\mu_m$ in ${\rm RKBM}_1$ solves the Petrov--Galerkin problem:
\begin{equation*}
(\mathbb{A}(\mu)\hat{\mathbf{u}}_\mu,\mathbf{v})=(\mathbf{f},\mathbf{v}),\quad \mathbf{v}\in\mathcal{L}^\mu_m.
\end{equation*}
The classical property of the BiCG implies that $\hat{\mathbf{u}}_\mu$ is also the $m$-th BiCG iterate of $\mathbb{B}\mathbb{A}(\mu)\mathbf{u}_\mu=\mathbb{B}\mathbf{f}$. Therefore, the error analysis of $\|\mathbf{u}_\mu - \hat{\mathbf{u}}_\mu\|_{\mathbb{M}}$ is equivalent to the convergence rate of the BiCG method.

Let $P_{\mathcal{K}^\mu_m}: \mathbb{R}^n\rightarrow\mathbb{R}^n$ be the Petrov--Galerkin projection onto $\mathcal{K}^\mu_m$, i.e., $P_{\mathcal{K}^\mu_m}\mathbf{v}\in\mathcal{K}^\mu_m$ is determined by
\[
(\mathbb{A}(\mu)P_{\mathcal{K}^\mu_m}\mathbf{v},\mathbf{w})=(\mathbb{A}(\mu)\mathbf{v},\mathbf{w}),\quad\forall\mathbf{w}\in\mathcal{K}^\mu_m.
\]
The conditions \eqref{infsup} ensure that 
\begin{equation*}
    \begin{aligned}
        &\beta(\mu)\|P_{\mathcal{K}^\mu_m}\mathbf{v}\|_{\mathbb{M}}\leq\sup_{\mathbf{0}\neq\mathbf{w}\in\mathbb{R}^n}\frac{(\mathbb{A}(\mu)P_{\mathcal{K}^\mu_m}\mathbf{v},\mathbf{w})}{\|\mathbf{w}\|_{\mathbb{M}^\prime}}\\
        &\quad=\sup_{\mathbf{0}\neq\mathbf{w}\in\mathbb{R}^n}\frac{(\mathbb{A}(\mu)\mathbf{v},\mathbf{w})}{\|\mathbf{w}\|_{\mathbb{M}^\prime}}\leq\alpha(\mu)\|\mathbf{v}\|_{\mathbb{M}}
    \end{aligned}
\end{equation*}
and thus 
\begin{equation}\label{Pkmbound}
\|P_{\mathcal{K}^\mu_m}\|_{\mathbb{M}}\leq\frac{\alpha(\mu)}{\beta(\mu)}.   
\end{equation}
Combining \eqref{Pkmbound} with the identity $\|I-P_{\mathcal{K}^\mu_m}\|_{\mathbb{M}}=\|P_{\mathcal{K}^\mu_m}\|_{\mathbb{M}}$ (cf.~\cite{XuZikatanov2003}), for any $\mathbf{v}\in\mathcal{K}_m^\mu$ we have 
\begin{equation*}
\begin{aligned}
        &\|\mathbf{u}_\mu - \hat{\mathbf{u}}_\mu\|_{\mathbb{M}}=\|\mathbf{u}_\mu - P_{\mathcal{K}^\mu_m}\mathbf{u}_\mu\|_{\mathbb{M}}\\
        &=\|(I-P_{\mathcal{K}^\mu_m})(\mathbf{u}_\mu - \mathbf{v})\|_{\mathbb{M}}\\
        &\leq\|I-P_{\mathcal{K}^\mu_m}\|\inf_{\mathbf{v}\in\mathcal{K}^\mu_m}\|\mathbf{u}_\mu - \mathbf{v}\|_{\mathbb{M}}\\
        &\leq\frac{\alpha(\mu)}{\beta(\mu)}\inf_{\mathbf{v}\in\mathcal{K}^\mu_m}\|\mathbf{u}_\mu - \mathbf{v}\|_{\mathbb{M}}.
    \end{aligned}
\end{equation*}
The proof is complete.
\end{proof}

Theorem \ref{errorRKBM2} states that ${\rm RKBM_2}$ achieves the best approximation error provided by the Krylov subspace $\mathcal{K}_m^\mu$ up to some multiplicative constant. A convergence estimate of the form \eqref{quasioptimal} with $\frac{\alpha(\mu)}{\beta(\mu)}$ replaced by $1+\frac{\alpha(\mu)}{\beta(\mu)}$ was first developed by \cite{BankChan1993} for a composite-step BiCG method. Let $\mathcal{P}_m^0$ denote the set of polynomials of degree $\leq m$ taking value 1 at 0. The estimate \eqref{quasioptimal} further implies  
\begin{equation}\label{BiCGerror}
      \|\mathbf{u}_\mu - \hat{\mathbf{u}}_\mu\|_{\mathbb{M}}\leq \frac{\alpha(\mu)}{\beta(\mu)}\inf_{\varphi\in\mathcal{P}_m^0}\|\varphi(\mathbb{M}^\frac{1}{2}\mathbb{B}\mathbb{A}(\mu)\mathbb{M}^{-\frac{1}{2}})\|\|\mathbf{u}_{\mu}\|_{\mathbb{M}}.
\end{equation}
As pointed out in \cite{BankChan1993}, there is no simple estimate for the infimum in \eqref{BiCGerror}. However, under mild assumptions, it is possible to bound $\inf_{\varphi\in\mathcal{P}_m^0}\|\varphi(\mathbb{M}^\frac{1}{2}\mathbb{B}\mathbb{A}(\mu)\mathbb{M}^{-\frac{1}{2}})\|$ in \eqref{BiCGerror} by Manteuffel's estimate, see, e.g., \cite{BankChan1993,Manteuffel1977} for details.

\section{Parametric PDEs with Multiple Parameters}\label{secmultipara}
In this section, we discuss variants of RKBMs that are more suitable for  the parametric variational problem \eqref{multiparabilinear} or \eqref{multiparamatrix}
with multiple parameters ($J>2$). 

For example, let $\{\Omega_j\}_{j=1}^J$ be a non-overlapping decomposition  of the domain $\Omega$, and $\alpha_\mu$ be a piecewise constant function with $$\mu=(\nu_1,\ldots,\nu_J),\qquad\alpha_\mu|_{\Omega_j}=\nu_j,$$see, e.g., Figure \ref{fig:pwcoeff} for an illustration. Let $\mathcal{H}\subset H_0^1(\Omega)$ be a conforming finite element subspace. We consider the second order elliptic equation
\begin{align*}
    -\nabla\cdot(\alpha_\mu\nabla \tilde{u}_\mu)&=f\quad\text{ in }\Omega,\\
    \tilde{u}_\mu&=0\quad\text{ on }\partial\Omega.
\end{align*}
The corresponding FEM is to find $u_\mu\in\mathcal{H}$ satisfying
\begin{equation}\label{multiparaelliptic}
\sum_{j=1}^J\nu_j\int_{\Omega_j}\nabla u_\mu\cdot\nabla vdx=\int_\Omega fvdx,\quad v\in\mathcal{H}.
\end{equation}
The above equation is a classical model problem of RBMs. Clearly \eqref{multiparaelliptic} is of the form
\eqref{multiparabilinear} with 
\begin{equation*}
   \theta_j(\mu)=\nu_j,\quad a_j(v,w)=\int_{\Omega_j}\nabla v\cdot\nabla wdx,\quad 1\leq j\leq4.
\end{equation*}

\begin{figure}[H]
    \centering
    \includegraphics[width=4.5cm,height=4.2cm]{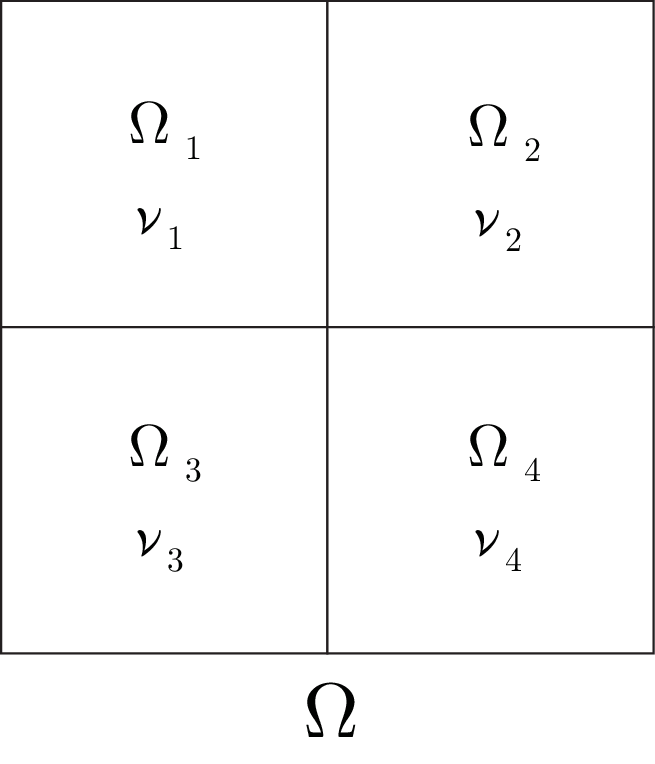}
    \caption{Piecewise constant diffusion coefficient $\alpha_\mu=\sum_{i=1}^4\nu_i\mathbbm{1}_{\Omega_i}$.}\label{fig:pwcoeff}
\end{figure}

Our error analysis in Sections \ref{secSPD} and \ref{secNSPD} only applies to the parametric PDE \eqref{multiparabilinear} with two parameter functions $\theta_1, \theta_2$. In practice, it turns out that the RKBMs based on a single Krylov subspace are not able to achieve high accuracy for \eqref{multiparabilinear} or \eqref{multiparamatrix}, see numerical experiments in Subsection \ref{subsecpwcoeff} for details. To improve the performance of the proposed algorithms, we choose to run Krylov subspace methods multiple times and use the sum of several Krylov subspaces as the reduced basis subspace. The corresponding generalized RKBMs are presented in Algorithms \ref{a:multiRKBM1} and \ref{a:multiRKBM2}.

\begin{algorithm}[H]
  \caption{${\rm mRKBM}_1(\mathbb{B},\mathbb{M}, \mathcal{P}, \mu_0, \mu_1,\ldots,\mu_L,  m)$ for solving $\mathbb{A}(\mu)\mathbf{u}_\mu=\mathbf{f}$}\label{a:multiRKBM1}
  \begin{algorithmic}
    \State \textbf{Input:} a parameter set $\mathcal{P}$ of practical interest,  $\mu_0, \mu_1, \ldots, \mu_L\in\mathcal{P}$ with $\mu_0\not\in\{\mu_1,\ldots,\mu_L\}$, an integer $m>0$, a SPD matrix $\mathbb{M}\in\mathbb{R}^{n\times n}$, a preconditioner $\mathbb{B}: \mathbb{R}^n\rightarrow\mathbb{R}^n$ such that $\mathbb{B}\approx\mathbb{A}(\mu_0)^{-1}$.
    \State \textbf{Offline Stage:} 
    \For{$l=1:L$}
    \State use the GMRES for $\mathbb{B}\mathbb{A}(\mu_l)\mathbf{u}_{\mu_l}=\mathbb{B}\mathbf{f}$ with initial guess $\mathbf{u}_{l,0}=\mathbf{0}$ to generate
    
    \State  iterative solutions $\mathbf{u}_{l,1}, \ldots, \mathbf{u}_{l,m-1}\approx\mathbf{u}_{\mu_l}$; 
    \State set  $\mathbf{z}_{l,j}=\mathbb{B}(\mathbf{f}-\mathbb{A}(\mu_l)\mathbf{u}_{l,j})$ for $0\leq j\leq m-1$;
   \EndFor
       
   \State construct the reduced basis subspace \[\mathcal{K}_m:= {\rm Span}\big\{\mathbf{z}_{l,j}\big\}_{1\leq l\leq L, 1\leq j\leq m}.\]

   \State \textbf{Online Stage:} for each $\mu\in\mathcal{P}$, compute $\hat{\mathbf{u}}_\mu\in\mathcal{K}_m$ by  
\begin{equation*}
\|\mathbb{B}\mathbf{f}-\mathbb{B}\mathbb{A}(\mu)\hat{\mathbf{u}}_\mu\|_{\mathbb{M}}=\min_{\mathbf{v}\in\mathcal{K}_m}\|\mathbb{B}\mathbf{f}-\mathbb{B}\mathbb{A}(\mu)\mathbf{v}\|_{\mathbb{M}}.
\end{equation*} 
  \end{algorithmic}  
\end{algorithm} 

\begin{algorithm}[H]
  \caption{${\rm mRKBM}_2(\mathbb{B}, \mathcal{P}, \mu_0, \mu_1,\ldots,\mu_L,  m)$ for solving $\mathbb{A}(\mu)\mathbf{u}_\mu=\mathbf{f}$}\label{a:multiRKBM2}
  \begin{algorithmic}
    \State Input: a parameter set $\mathcal{P}$ of practical interest,  $\mu_0, \mu_1, \ldots, \mu_L\in\mathcal{P}$ with $\mu_0\not\in\{\mu_1,\ldots,\mu_L\}$, an integer $m>0$, a preconditioner $\mathbb{B}: \mathbb{R}^n\rightarrow\mathbb{R}^n$ for $\mathbb{A}(\mu_0)$.
    \State \textbf{Offline Stage:} 
    \For{$l=1:L$}
    \State use the BiCG for $\mathbb{B}\mathbb{A}(\mu_l)\mathbf{u}_{\mu_l}=\mathbb{B}\mathbf{f}$ and $\mathbb{B}^t\mathbb{A}(\mu_l)^t\mathbf{u}^*_{\mu_l}=\mathbb{B}^t\mathbf{r}_0^*$ with initial guesses \State $\mathbf{u}_{l,0}=\mathbf{0}$, $\mathbf{u}^*_{l,0}=\mathbf{0}$  to generate iterative solutions $\mathbf{u}_{l,1}, \ldots, \mathbf{u}_{l,m-1}\approx\mathbf{u}_{\mu_l}$, \State $\mathbf{u}^*_{l,1}, \ldots, \mathbf{u}^*_{l,m-1}\approx\mathbf{u}^*_{\mu_l}$; for $j=0, 1, \ldots, m$, set $$\mathbf{z}_{l,j}=\mathbb{B}(\mathbf{f}-\mathbb{A}(\mu_l)\mathbf{u}_{l,j}),\quad\mathbf{z}^*_{l,j}=\mathbb{B}^t(\mathbf{r}_0^*-\mathbb{A}(\mu_l)^t\mathbf{u}^*_{l,j});$$
   \EndFor
       
   \State construct the reduced basis subspace \[\mathcal{K}_m:= {\rm Span}\big\{\mathbf{z}_{l,j}\big\}_{1\leq l\leq L, 1\leq j\leq m},\quad \mathcal{L}_m:= {\rm Span}\big\{\mathbf{z}^*_{l,j}\big\}_{1\leq l\leq L, 1\leq j\leq m}.\]

   \State \textbf{Online Stage:} for each $\mu\in\mathcal{P}$, compute $\hat{\mathbf{u}}_\mu\in\mathcal{K}_m$ by 
\begin{equation*}
(\mathbb{A}(\mu)\hat{\mathbf{u}}_\mu,\mathbf{v})=(\mathbf{f},\mathbf{v}),\quad\mathbf{v}\in\mathcal{L}_m.
\end{equation*} 
  \end{algorithmic}  
\end{algorithm} 

When $\mathbb{A}(\mu)$ is SPD, the algorithm ${\rm mRKBM}_2$ reduces to a variant of the RCGBM based on PCG iterations for $\mathbb{B}\mathbb{A}(\mu)\mathbf{u}_\mu=\mathbb{B}\mathbf{f}$ at multiple parameter instances $\mu=\mu_1, \ldots, \mu_L$, see Algorithm \ref{a:multiRCGBM}. 

\begin{algorithm}[H]
  \caption{${\rm mRCGBM}(\mathbb{B}, \mu_0, \mu_1,\ldots,\mu_L,  m)$ for solving $\mathbb{A}(\mu)\mathbf{u}_\mu=\mathbf{f}$}\label{a:multiRCGBM}
  \begin{algorithmic}
    \State \textbf{Input:} a parameter set $\mathcal{P}$ of practical interest,  $\mu_0, \mu_1, \ldots, \mu_L\in\mathcal{P}$ with $\mu_0\not\in\{\mu_1,\ldots,\mu_L\}$, an integer $m>0$, a preconditioner $\mathbb{B}: \mathbb{R}^n\rightarrow\mathbb{R}^n$ for $\mathbb{A}(\mu_0)$.
    \State \textbf{Offline Stage:} 
    \For{$l=1:L$}
    \State set 
    $\mathbf{u}_{l,0}=\mathbf{0}$,
    $\mathbf{r}_{l,0}=\mathbf{f}$, $\mathbf{p}_{l,0}=\mathbf{z}_{l,0}=\mathbb{B}\mathbf{r}_{l,0}$;
 use the PCG for $\mathbb{B}\mathbb{A}(\mu_l)\mathbf{u}_{\mu_l}=\mathbb{B}\mathbf{f}$ with \State initial guess  $\mathbf{u}_{l,0}=\mathbf{0}$ to generate conjugate gradients $\mathbf{p}_{l,0}, \ldots, \mathbf{p}_{l,m-1}$; set \begin{align*}
       \mathcal{K}_m:={\rm Span}\left\{\frac{\mathbf{p}_{1,0}}{\|\mathbf{p}_{1,0}\|},\dots,\frac{\mathbf{p}_{1,m-1}}{\|\mathbf{p}_{1,m-1}\|},\ldots,\frac{\mathbf{p}_{L,0}}{\|\mathbf{p}_{L,0}\|},\dots,\frac{\mathbf{p}_{L,m-1}}{\|\mathbf{p}_{L,m-1}\|}\right\}.
   \end{align*}
   \EndFor
   \State \textbf{Online Stage:} for each $\mu\in\mathcal{P}$, find the reduced basis solution $\hat{\mathbf{u}}_\mu\in\mathcal{K}_m$ satisfying 
\begin{equation*}
(\mathbb{A}(\mu)\hat{\mathbf{u}}_\mu,\mathbf{v})=(\mathbf{f},\mathbf{v}),\quad \mathbf{v}\in\mathcal{K}_m.
\end{equation*} 
  \end{algorithmic}  
\end{algorithm} 

We remark that the spanning vectors in Algorithms \ref{a:multiRKBM1}, \ref{a:multiRKBM2}, and \ref{a:multiRCGBM} do not form a basis of the reduced basis subspace $\mathcal{K}_m$. However, it is easy to compute the reduced basis solution $\hat{\mathbf{u}}_\mu$ by solving an $m\times m$ small-scale least-squares problem.

\section{Numerical Experiments}\label{secNE}
To validate the efficiency of the proposed reduced Krylov basis methods, we test the performance of the RCGBM, ${\rm RKBM}_1$, ${\rm RKBM}_2$ for a series of parametric PDEs including  Maxwell equations, Navier--Lam\'e equations, convection-diffusion equations, and second order elliptic equations with piecewise coefficients. 

In each experiment, our 2-dimensional mesh refinement is through dividing each triangle into four similar sub-triangles. Similarly, 3-dimensional meshes are generated by successive regular refinement in the sense of \cite{Bey2000}, that is, dividing each tetrahedral element into eight sub-elements. In each tables, by $nv, ne$ we denote the number of vertices and number of edges in the mesh, respectively.
All numerical experiments are performed in MATLAB 2022b.

\subsection{H(curl) Elliptic Equation}\label{subsec:cube}
The first experiment is concerned with the eddy current Maxwell equation explained in Subsection \ref{subsecMaxwell}. We set $\Omega=[0,1]^3$, $\bm{f}=(10,10,10)^t$, and use the lowest order N\'ed\'elec edge element to construct the high-fidelity approximation $\bm{u}_\mu$ for each $\mu\in\mathcal{P}=(1:0.4:3)\times(1:0.4:3)$. In this case, $ne$ is the number of unknowns in the finite element linear systems of equations.

Then ${\rm RCGBM}(\mathbb{B},\mathcal{P},\mu_0,\mu_1,m)$ is applied to the edge element discretization, where $\mu_0=(1,1)$, $\mu_1=(1,2)$, and $\mathbb{B}$ is the Hiptmair--Xu (HX) preconditioner (cf.~\cite{HiptmairXu2007,Li2023FoCM}) for $\mathbb{A}(\mu_1)$ such that $\mathbb{B}\approx\mathbb{A}(\mu_1)^{-1}$. 
Here we use one algebraic multigrid (AMG) W-cycle as the Poisson solvers to construct the HX preconditioner for the Maxwell $H({\rm curl})$ systems. Numerical results are presented in Figure and Table \ref{tab:errorMaxwell}, where all errors are measured in the $H({\rm curl})$ norm $$\|\bm{v}\|_{H({\rm curl})}:=\big(\|\bm{v}\|^2_{L^2(\Omega)}+\|\nabla\times\bm{v}\|^2_{L^2(\Omega)}\big)^\frac{1}{2}.$$

It is observed from Table \ref{tab:errorMaxwell} and Figure \ref{fig:errorMaxwell} that as the number $m$ of conjugate gradient reduced basis grows, the error $\sup_{\mu\in\mathcal{P}}\|\bm{u}_\mu-\hat{\bm{u}}_\mu\|_{H({\rm curl})}$ is decreasing. In particular, a small $m\approx10$ is able to guarantee high accuracy of the RCGBM. Figure \ref{tab:errorMaxwell} shows that the reduced basis solution $\hat{\bm{u}}_\mu$ remains uniformly accurate for varying $\mu\in\mathcal{P}$. We remark that ${\rm RCGBM}(\mathbb{B},\mathcal{P},\mu_0,\mu_1,m)$ called the PCG  with the HX preconditioner $\mathbb{B}$ only one time. In contrast, directly solving $\mathbb{A}(\mu)\mathbf{u}_\mu=\mathbf{f}$ for all $\mu\in\mathcal{P}$ requires running the PCG function $\#\mathcal{P}=25$ times and thus is much slower, especially for large-scale problems, e.g., $ne=1872064$.

\begin{figure}[H]
    \centering
\includegraphics[width=10cm,height=6cm]{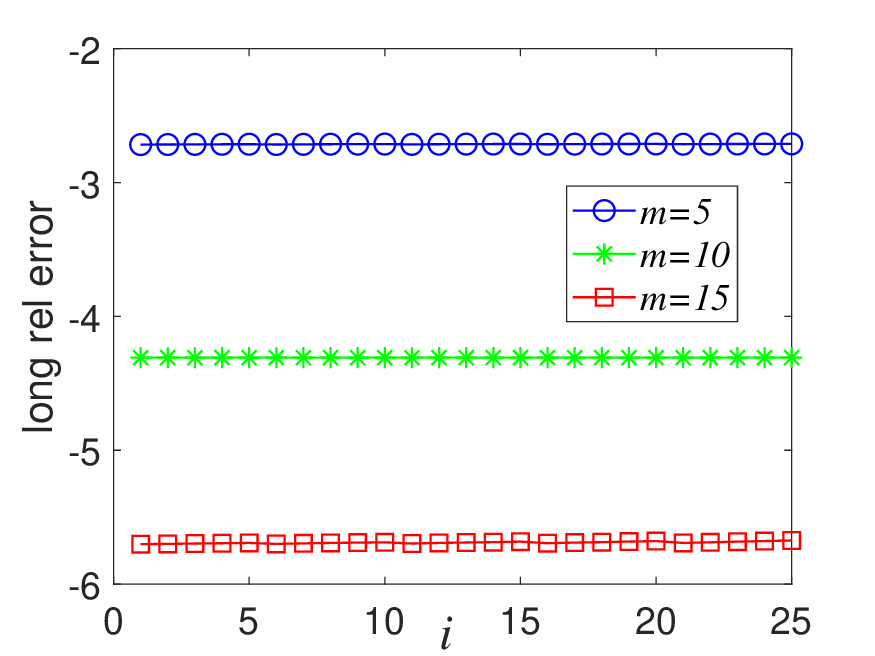}
    \caption{$\log_{10}(\|\bm{u}_{\tilde{\mu}_i}-\hat{\bm{u}}_{\tilde{\mu}_i}\|_{H({\rm curl})}/\|\bm{u}_{\tilde{\mu}_i}\|_{H({\rm curl})})$ of ${\rm RCGBM}(\mathbb{B},\mathcal{P},(1,1),(1,2),m)$ for  $\tilde{\mu}_i\in\mathcal{P}$, where $\tilde{\mu}_1=(1,1), \ldots, \tilde{\mu}_{25}=(3,3)$ are ordered in lexicographic way, $ne=1872064$.}
    \label{fig:errorMaxwell}
  \end{figure}
  
\begin{table}[H]
  \centering
  \begin{tabular}{ |r||p{1.7cm}|p{1.7cm}| p{1.7cm}|p{1.7cm}| }
    \hline
    \hline
    $ne$ & $~~m=5$ & $~~m=10$ & $~~~m=15$\\
    \hline
    4184     & 5.3632e-4 & 5.1116e-6  &~~4.9147e-8\\
    31024     & 1.3139e-3 &  9.5859e-6 &~~1.0647e-7\\
    238688     & 1.5370e-3 &  2.0646e-5 &~~6.9724e-7\\
    1872064     & 1.9448e-3 & 4.9251e-5  &~~2.1185e-6\\
    \hline
   \end{tabular}
   \caption{$\sup_{\mu\in\mathcal{P}}\|\bm{u}_\mu-\hat{\bm{u}}_\mu\|_{H({\rm curl})}/\|\bm{u}_\mu\|_{H({\rm curl})}$ of ${\rm RCGBM}(\mathbb{B},\mathcal{P},(1,1),(1,2),m)$.}\label{tab:errorMaxwell}
   \end{table}

\subsection{Convection-Diffusion Problem}
In the second example, we use the linear Lagrange element to compute the high-fidelity numerical solution $u_\mu$  of  the convection-diffusion equation in Subsection \ref{subsecConvectionDiffusion} with $\bm{b}=(1,-2)$ and $f=10$ on $\Omega=[-1,1]^2$. Let $\mathbb{A}_1$ be the linear finite element stiffness matrix of the Poisson equation. We run ${\rm RKBM}_1(\mathbb{B},\mathbb{M},\mathcal{P},\mu_0,\mu_1,m)$ and ${\rm RKBM}_2(\mathbb{B},\mathcal{P},\mu_0,\mu_1,m)$ for this problem with $\mathbb{B}=\mathbb{A}^{-1}_1$, $\mathbb{M}=\mathbb{A}_1$, $\mathcal{P}=(0.4:0.4:2)\times(0:\frac{2\pi}{5}:2\pi)$
to compute the reduced basis solutions $\hat{u}_\mu$ for each $\mu\in\mathcal{P}$.

As shown in Figures \ref{fig:errorConvectionRKBM1}, \ref{fig:errorConvectionRKBM2} and Tables \ref{tab:errorConvectionRKBM1}, \ref{tab:errorConvectionRKBM2} that algorithms ${\rm RKBM}_1$ and ${\rm RKBM}_2$ maintain quite similar level of accuracy. It is also observed from Figures \ref{fig:errorConvectionRKBM1} and \ref{fig:errorConvectionRKBM2} that the numerical error $|u_\mu-\hat{u}_\mu|_{H^1(\Omega)}$ converges very fast as the number $m$ of Krylov basis grows. In addition,  we note that $\hat{u}_\mu$ is more accurate for those $\mu=(\nu_1,\nu_2)$ with larger $\nu_1$ and small $|\cos(\nu_2)|$. The reason is that the preconditioner $\mathbb{B}$ is more efficient for diffusion-dominated problems. 
 \begin{figure}[H]
    \centering
\includegraphics[width=10cm,height=6cm]{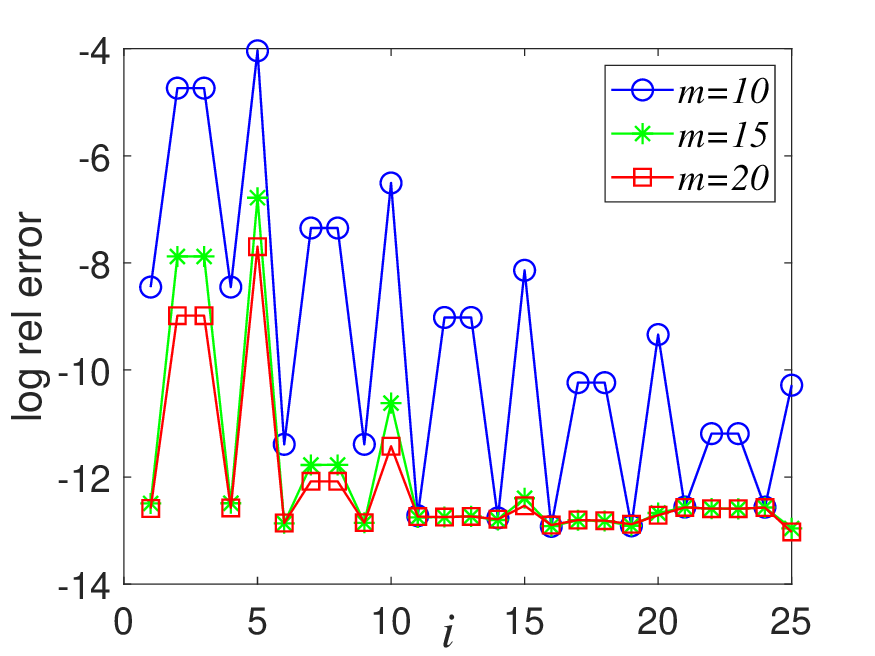}
    \caption{$\log_{10}(|u_{\tilde{\mu}_i}-\hat{u}_{\tilde{\mu}_i}|_{H^1(\Omega)}/|u_{\tilde{\mu}_i}|_{H^1(\Omega)})$ of ${\rm RKBM}_1(\mathbb{B},\mathbb{M},\mathcal{P},(1,\frac{\pi}{2}),(1,0),m)$ for   $\tilde{\mu}_i\in\mathcal{P}$, where $\tilde{\mu}_1=(0.4,0), \ldots, \tilde{\mu}_{25}=(2,2\pi)$ are ordered in lexicographic way, $nv=263169$.}
    \label{fig:errorConvectionRKBM1}
  \end{figure}

\begin{table}[H]
  \centering
  \begin{tabular}{ |r||p{1.9cm}|p{1.9cm}|p{1.9cm}|p{1.9cm} }
    \hline
    \hline
    $nv$ & $~~~m=10$ & $~~~m=15$ & $~~~~m=20$ \\
    \hline
    4225     &~~8.9744e-5 &~~1.7730e-7  &~~~9.1939e-9 \\
    16641     &~~9.0955e-5 &~~1.6408e-7 &~~~9.4322e-9\\
    66049     &~~9.1260e-5 &~~1.6441e-7 &~~~1.2664e-8 \\
    263169     &~~9.1336e-5 &~~1.6472e-7  &~~~2.0095e-8 \\
    \hline
   \end{tabular}
   \caption{$\sup_{\mu\in\mathcal{P}}|u_\mu-\hat{u}_\mu|_{H^1(\Omega)}/|u_\mu|_{H^1(\Omega)}$ of ${\rm RKBM}_1(\mathbb{B},\mathbb{M},\mathcal{P},(1,\frac{\pi}{2}),(1,0),m)$.}\label{tab:errorConvectionRKBM1}
   \end{table}

\begin{figure}[H]
    \centering
\includegraphics[width=10cm,height=6cm]{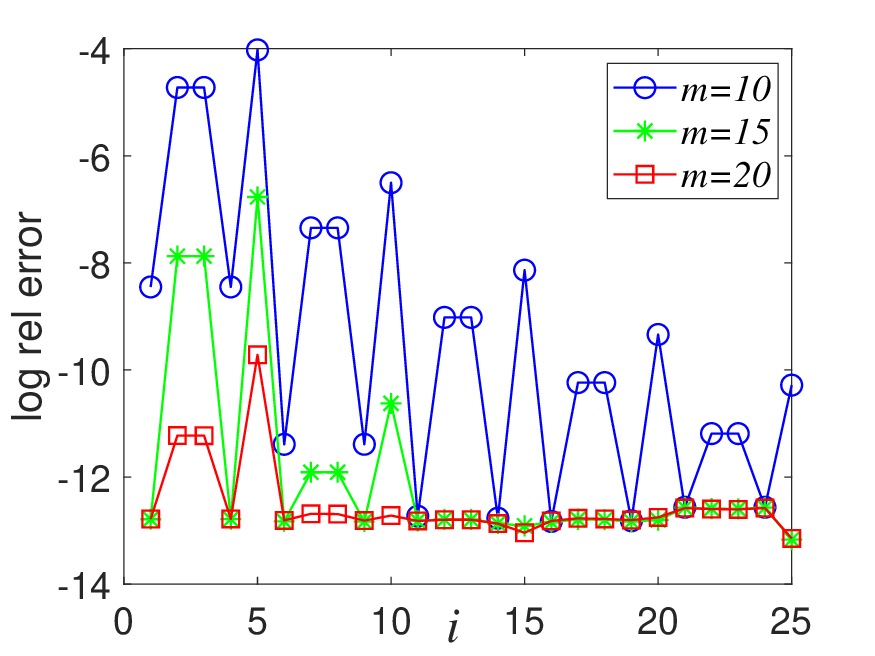}
    \caption{$\log_{10}(|u_{\tilde{\mu}_i}-\hat{u}_{\tilde{\mu}_i}|_{H^1(\Omega)}/|u_{\tilde{\mu}_i}|_{H^1(\Omega)})$ of ${\rm RKBM}_2(\mathbb{B},\mathcal{P},(1,\frac{\pi}{2}),(1,0),m)$ for  $\tilde{\mu}_i\in\mathcal{P}$, where $\tilde{\mu}_1=(0.4,0), \ldots, \tilde{\mu}_{25}=(2,2\pi)$ are ordered in lexicographic way, $nv=263169$.}
    \label{fig:errorConvectionRKBM2}
  \end{figure}

\begin{table}[H]
  \centering
  \begin{tabular}{ |r||p{1.9cm}|p{1.9cm}|p{1.9cm}|p{1.9cm} }
    \hline
    \hline
    $nv$ & $~~~m=10$ & $~~~m=15$ & $~~~~m=20$ \\
    \hline
    4225     &~~9.3232e-5 &~~1.8403e-7  &~~~6.4488e-10 \\
    16641     &~~9.4535e-5 &~~1.6982e-7 &~~~6.6489e-10\\
    66049     &~~9.4850e-5 &~~1.6981e-7 &~~~4.9609e-10 \\
    263169     &~~9.4968e-5 &~~1.7000e-7  &~~~1.9145e-10 \\
    \hline
   \end{tabular}
   \caption{$\sup_{\mu\in\mathcal{P}}|u_\mu-\hat{u}_\mu|_{H^1(\Omega)}/|u_\mu|_{H^1(\Omega)}$ of ${\rm RKBM}_2(\mathbb{B},\mathcal{P},(1,\frac{\pi}{2}),(1,0),m)$.}\label{tab:errorConvectionRKBM2}
   \end{table}

\subsection{Linear Elasticity}
The third example is devoted to  the Navier--Lamé equation for linear elasticity, as introduced in Subsection \ref{subseclinearelasticity}, with $\bm{f}=(10,10,10)^t$ on the unit cube $\Omega=[0,1]^3$. We use the linear Lagrange element to compute the FEM solution $\bm{u}_{\mu}\in\mathcal{H}\subset [H_0^1(\Omega)]^3$ for $\mu = (\nu_1, \nu_2) \in \mathcal{P} = 1 \times (0.05:0.01:0.3).$ All errors are measured in the energy norm $$|\!|\!|\bm{v}|\!|\!| = \sqrt{a_\mu(\bm{v},\bm{v})},\quad\bm{v}\in\mathcal{H},$$
where $a_{\mu}(\cdot, \cdot)$ is defined in \eqref{linearElasticity}.

Let $e_i$ denote the $i$-th unit vector in $\mathbb{R}^3$ and $\{\phi_k\}$ be a set of scalar-valued nodal basis functions.
The stiffness matrix $\mathbb{A}(\mu)=(\mathbb{A}_{ij}(\mu))_{1\leq i,j\leq3}$ is a $3\times3$ block diagonal matrix, where the $(k,l)$-entry of the block $\mathbb{A}_{ij}(\mu)$ is $a_\mu(\phi_l e_i,\phi_ke_j)$.
First we investigate the performance of ${\rm RCGBM}(\mathbb{B},\mathcal{P},\mu_0,\mu_1,24)$ for linear elasticity, where $\mu_0 = (1, 0.05)$, $\mu_1 = (1, 0.1)$, and $\mathbb{B}$ is a block diagonal preconditioner given by
\begin{equation*}
    \mathbb{B} = {\rm blkdiag}(\mathbb{A}_{11}(\mu_0)^{-1}, \mathbb{A}_{22}(\mu_0)^{-1}, \mathbb{A}_{33}(\mu_0)^{-1}).
\end{equation*}
Simple calculation shows that each $\mathbb{A}_{ii}(\mu_0)$ is spectrally equivalent to the discrete Laplacian from linear finite elements.
Although $\mathbb{B}$ is an efficient preconditioner in the framework of operator preconditioning (cf.~\cite{LoghinWathen2004,MardalWinther2011}), as illustrated in Figure \ref{fig:elasticity}, the numerical error $|\!|\!|u_{\mu} - \hat{u}_{\mu}|\!|\!|$ does not meet the desired accuracy. 

To improve the accuracy, we then call ${\rm mRCGBM}(\mathbb{B},\mathcal{P},\mu_0,\mu_1,\ldots,\mu_L,m)$ in Section \ref{secmultipara}  and construct the reduced basis subspace as the sum of several Krylov subspaces. In particular, we run ${\rm mRCGBM}(\mathbb{B},\mathcal{P},\mu_0,\mu_1,\ldots,\mu_L,m)$ with $\mu_2 = (1, 0.25)$, $\mu_3 = (1, 0.3)$ and $L=2, 3$. When $L=1$, ${\rm mRCGBM}(\mathbb{B},\mathcal{P},\mu_0,\mu_1,\ldots,\mu_L,m)$ reduces to ${\rm RCGBM}(\mathbb{B},\mathcal{P},\mu_0,\mu_1,m)$. 
From Figure \ref{fig:elasticity}, it is evident that the RBMs based on the sum of multiple Krylov subspaces outperform the RCGBM in a single Krylov subspace.
As shown in Figure \ref{fig:elasticity} and Table \ref{tab:errorElasticity}, the strategy of constructing the reduced basis subspace with additional Krylov spaces is efficient; it allows a small subspace to provide a solution with reasonable accuracy.

\begin{figure}[H]
    \centering
    \includegraphics[width=10cm,height=6cm]{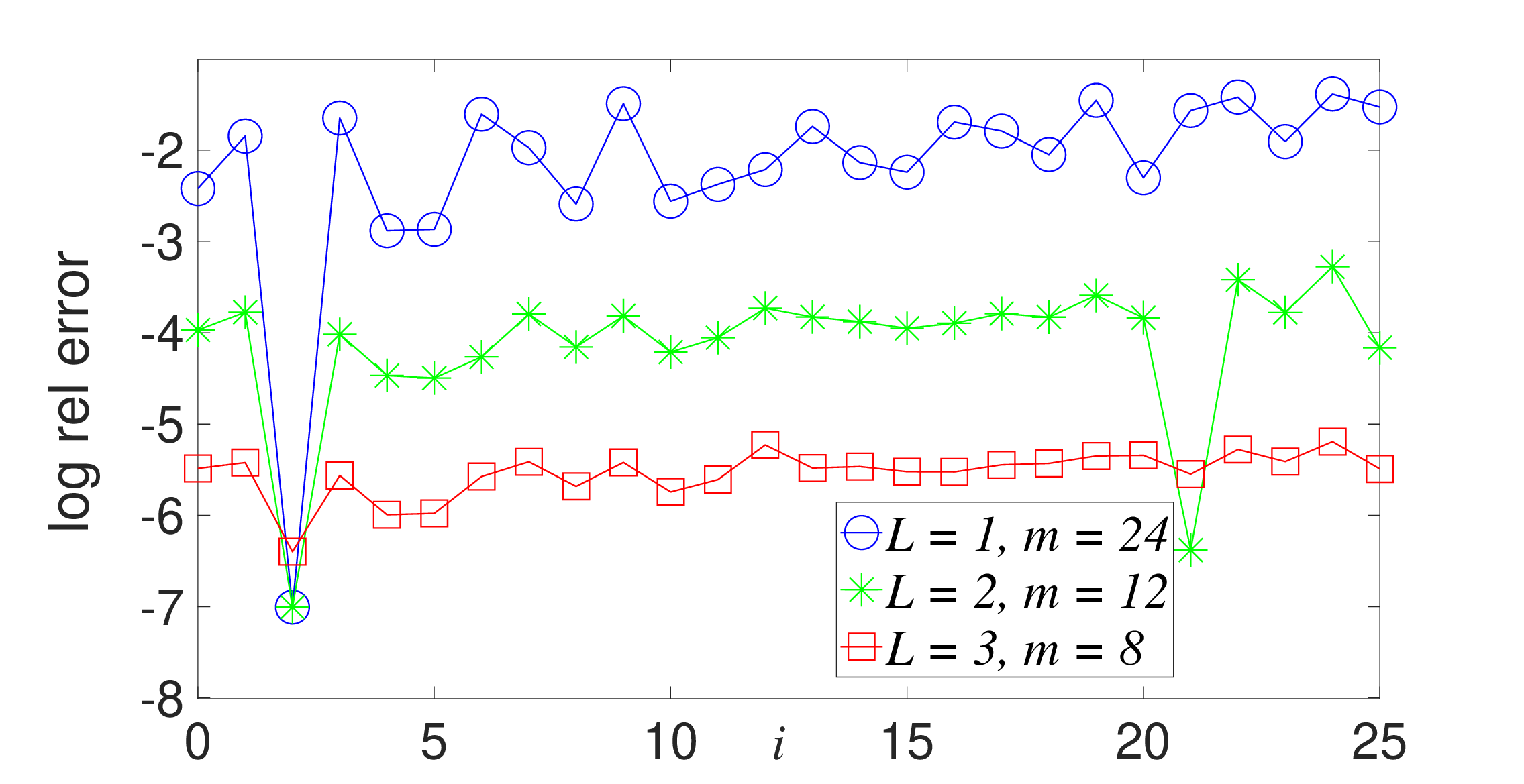}
    \label{fig:elasticity}
    \caption{$\log_{10}(|\!|\!|u_{\tilde{\mu}_i}-\hat{u}_{\tilde{\mu}_i}|\!|\!|/|\!|\!|u_{\tilde{\mu}_i}|\!|\!|)$ of ${\rm mRCGBM}(\mathbb{B},\mathcal{P},\mu_0,\mu_1,24)$, \\${\rm mRCGBM}(\mathbb{B}, \mathcal{P}, \mu_0,\mu_1,\mu_2, 12)$, ${\rm mRCGBM}(\mathbb{B},\mathcal{P},\mu_0,\mu_1,\mu_2,\mu_3, 4)$, where $\tilde{\mu}_i \in \mathcal{P}$, $nv = 274625$.}
    
\end{figure}

   \begin{table}[H]
  \centering
  \begin{tabular}{ |r||p{1.9cm}|p{1.9cm}|p{1.9cm}|p{1.9cm} }
    \hline
    \hline
    $nv$ & $~~~m=4$ & $~~~m=6$ & $~~~~m=8$ \\
    \hline
    4913     &~~7.9519e-4 &~~6.7498e-5  &~~~6.7275e-6 \\
    35937     &~~8.9282e-4 &~~7.3524e-5 &~~~6.4003e-6\\
    274625     &~~9.1439e-4 &~~7.7088e-5 &~~~6.4091e-6\\
    \hline
   \end{tabular}
   \caption{$\sup_{\mu\in\mathcal{P}}|\!|\!|u_\mu-\hat{u}_\mu|\!|\!|/|\!|\!|u_\mu|\!|\!|$ of ${\rm mRCGBM}(\mathbb{B},\mathcal{P}_0, \mu_0, \mu_1, \mu_2, \mu_3, m)$, $m = 4,6,8$.}
   \label{tab:errorElasticity}
   \end{table}

\subsection{Elliptic Equations with Piecewise Constant Coefficients}\label{subsecpwcoeff}
In the last experiment, we consider the linear finite element method \eqref{multiparaelliptic} for the second order elliptic equation based on a uniform grid of $\Omega=[-1,1]\times[-1,1]$ with 1050625 vertices. The source term $f=10$ and the coefficient $\alpha_\mu$ with $\mu=(\nu_1,\nu_2,\nu_3,\nu_4)$ is given in Figure \ref{fig:pwcoeff}. Let $\mu_0=(1,1,1,1)$, $\mu_1=(1,2,3,4)$, $\mu_2=(1,2,1,4)$, $\mu_3=(1,1,2,4)$ and $\mathcal{P}=(1:1:3)\times(1:1:3)\times(1:1:3)\times(1:1:3)$. Setting  $\mathbb{B}=\mathbb{A}(\mu_0)^{-1}$ as the  preconditioner, we run ${\rm mRCGBM}(\mathbb{B},\mathcal{P},\mu_0,\ldots,\mu_L,m)$ for \eqref{multiparaelliptic} with $\mu\in\mathcal{P}$, $m=5, 10, 15$, and $L=1, 2, 3$.

\begin{table}[H]
  \centering
  \begin{tabular}{ |r||p{1.9cm}|p{1.9cm}|p{1.9cm}|p{1.9cm} }
    \hline
    \hline
    $L$ & $~~~m=5$ & $~~~m=10$ & $~~~~m=15$ \\
    \hline
    1     &~~3.4926e-1 &~~2.4486e-1  &~~~2.4492e-1 \\
    2     &~~1.1540e-1 &~~4.0727e-3 &~~~2.6151e-4\\
    3     &~~2.7279e-2 &~~1.2793e-6 &~~~2.1137e-8\\
    \hline
   \end{tabular}
   \caption{$\sup_{\mu\in\mathcal{P}_0}|u_\mu-\hat{u}_\mu|_{H^1(\Omega)}/|u_\mu|_{H^1(\Omega)}$ of ${\rm mRCGBM}(\mathbb{B},\mathcal{P}_0,\mu_0,\mu_1,\ldots,\mu_L,m)$.}\label{tab:errorpwcoeffRCGBM}
   \end{table}

\begin{figure}[H]
    \centering
\includegraphics[width=10cm,height=6cm]{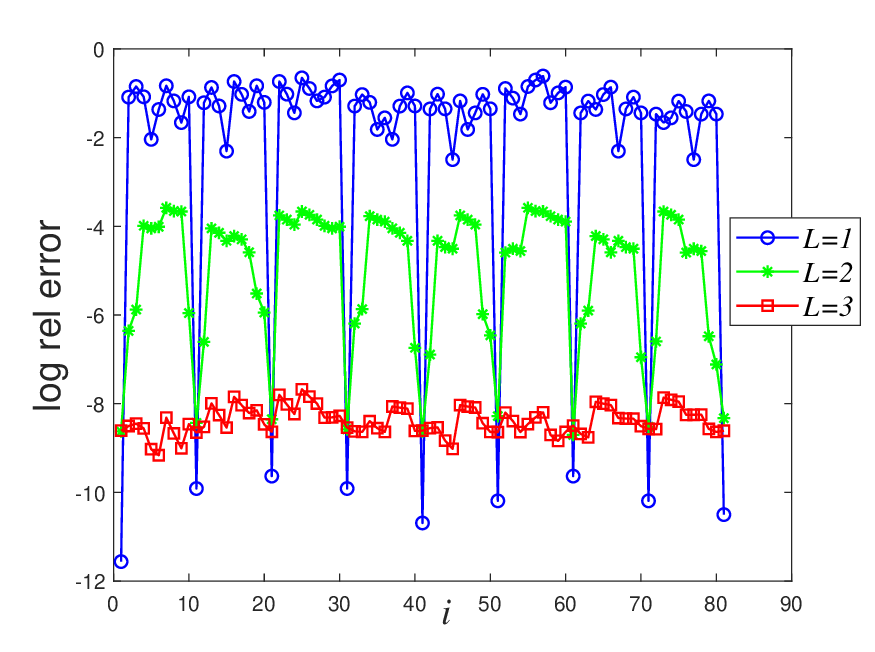}
    \caption{$\log_{10}(|u_{\tilde{\mu}_i}-\hat{u}_{\tilde{\mu}_i}|_{H^1(\Omega)}/|u_{\tilde{\mu}_i}|_{H^1(\Omega)})$ of ${\rm mRCGBM}(\mathbb{B},\mathcal{P}_0,\mu_0,\ldots,\mu_L,15)$ for  $\tilde{\mu}_i\in\mathcal{P}$, where $\tilde{\mu}_1=(1,1,1,1), \ldots, \tilde{\mu}_{25}=(3,3,3,3)$ are ordered in lexicographic way, $nv=1050625$.}
    \label{fig:errorpwcoeffRCGBM}
  \end{figure}

It is observed from Table \ref{tab:errorMaxwell} that ${\rm mRCGBM}(\mathbb{B},\mathcal{P},\mu_0,\mu_1,\mu_2,\mu_3,m)$ built upon three Krylov spaces converges fast. In comparison, ${\rm mRCGBM}(\mathbb{B},\mathcal{P},\mu_0,\mu_1,\mu_2,m)$ is able to maintain reasonable accuracy, while ${\rm mRCGBM}(\mathbb{B},\mathcal{P},\mu_0,\mu_1,m)$ is not convergent at all as $m$ grows. As shown in Figure \ref{fig:errorpwcoeffRCGBM}, ${\rm mRCGBM}(\mathbb{B},\mathcal{P},\mu_0,\mu_1,\mu_2,\mu_3,m)$ is uniformly accurate for all $\mu\in\mathcal{P}$. On the other hand, the error $|u_\mu-\hat{u}_\mu|_{H^1(\Omega)}$ is only accidentally small for a few parameters $\mu\in\mathcal{P}$.

\bibliographystyle{siamplain}

\end{document}